\documentclass[11pt,reqno]{amsart}

\usepackage[margin=1in]{geometry} 
\usepackage{leftidx}
\usepackage{amssymb,amsfonts,amsmath,bbm,mathrsfs,stmaryrd}
\usepackage{xcolor}
\usepackage{url}

\usepackage{a4wide}
\usepackage{dsfont} 
\usepackage{graphicx}

\usepackage[colorlinks,
             linkcolor=blue,
             citecolor=black!75!red,
             pdfproducer={pdfLaTeX},
             pdfpagemode=None,
             bookmarksopen=true
             bookmarksnumbered=true]{hyperref}

\usepackage{tikz}
\usetikzlibrary{arrows,calc,decorations.pathreplacing,decorations.markings,intersections,shapes.geometric,through,fit,shapes.symbols,positioning,decorations.pathmorphing}

\usepackage{amsmath,amsthm,stmaryrd}
\usepackage{cleveref}

\newtheorem{lemma}{Lemma}[section]
\newtheorem{theorem}[lemma]{Theorem}
\newtheorem{proposition}[lemma]{Proposition}
\newtheorem{corollary}[lemma]{Corollary}

\theoremstyle{definition} 

\newtheorem{definitionnodiamond}[lemma]{Definition}
\newtheorem{examplenodiamond}[lemma]{Example}
\newtheorem{remarknodiamond}[lemma]{Remark}

\newenvironment{definition}{\begin{definitionnodiamond}}{\hfill\ensuremath\blacklozenge\end{definitionnodiamond}}
\newenvironment{example}{\begin{examplenodiamond}}{\hfill\ensuremath\blacklozenge\end{examplenodiamond}}
\newenvironment{remark}{\begin{remarknodiamond}}{\hfill\ensuremath\blacklozenge\end{remarknodiamond}}

\makeatletter
\let\xx@thm\@thm
\AtBeginDocument{\let\@thm\xx@thm}
\makeatother

\renewcommand\qedhere{\qed}

\numberwithin{equation}{section}

\crefname{section}{Section}{Sections}
\crefformat{section}{#2Section~#1#3} 
\Crefformat{section}{#2Section~#1#3} 

\crefname{subsection}{}{Subsections}
\crefformat{subsection}{\S#2#1#3} 
\Crefformat{subsection}{\S#2#1#3}

\crefname{definition}{Definition}{Definitions}
\crefformat{definition}{#2Definition~#1#3} 
\Crefformat{definition}{#2Definition~#1#3} 

\crefname{definitionnodiamond}{Definition}{Definitions}
\crefformat{definitionnodiamond}{#2Definition~#1#3} 
\Crefformat{definitionnodiamond}{#2Definition~#1#3}

\crefname{example}{Example}{Examples}
\crefformat{example}{#2Example~#1#3} 
\Crefformat{example}{#2Example~#1#3} 

\crefname{examplenodiamond}{Example}{Examples}
\crefformat{examplenodiamond}{#2Example~#1#3} 
\Crefformat{examplenodiamond}{#2Example~#1#3} 

\crefname{remark}{Remark}{Remarks}
\crefformat{remark}{#2Remark~#1#3} 
\Crefformat{remark}{#2Remark~#1#3} 

\crefname{remarknodiamond}{Remark}{Remarks}
\crefformat{remarknodiamond}{#2Remark~#1#3} 
\Crefformat{remarknodiamond}{#2Remark~#1#3} 

\crefname{convention}{Convention}{Conventions}
\crefformat{convention}{#2Convention~#1#3} 
\Crefformat{convention}{#2Convention~#1#3} 

\crefname{lemma}{Lemma}{Lemmas}
\crefformat{lemma}{#2Lemma~#1#3} 
\Crefformat{lemma}{#2Lemma~#1#3} 

\crefname{proposition}{Proposition}{Propositions}
\crefformat{proposition}{#2Proposition~#1#3} 
\Crefformat{proposition}{#2Proposition~#1#3} 

\crefname{corollary}{Corollary}{Corollaries}
\crefformat{corollary}{#2Corollary~#1#3} 
\Crefformat{corollary}{#2Corollary~#1#3} 

\crefname{theorem}{Theorem}{Theorems}
\crefformat{theorem}{#2Theorem~#1#3} 
\Crefformat{theorem}{#2Theorem~#1#3} 

\crefname{assumption}{Assumption}{Assumptions}
\crefformat{assumption}{#2Assumption~#1#3} 
\Crefformat{assumption}{#2Assumption~#1#3} 

\crefname{equation}{}{}
\crefformat{equation}{(#2#1#3)} 
\Crefformat{equation}{(#2#1#3)}

\crefname{align}{}{}
\crefformat{align}{(#2#1#3)} 
\Crefformat{align}{(#2#1#3)}

\crefname{proofstep}{Step}{Steps}
\crefformat{proofstep}{#2Step~#1#3} 
\Crefformat{proofstep}{#2Step~#1#3}

\newcommand\cat[1]{\textsc{#1}}
\newcommand\define[1]{\emph{#1}}

\newcommand{\cst}{\ifmmode\mathrm{C}^*\else{$\mathrm{C}^*$}\fi}
\newcommand{\wst}{\ifmmode\mathrm{C}^*\else{$\mathrm{W}^*$}\fi}

\newcommand{\id}{\mathrm{id}}
\newcommand{\ad}{Ad}
\newcommand{\cad}{ad}

\newcommand{\I}{\mathds{1}}

\newcommand{\GG}{\mathbb{G}}

\newcommand{\op}{{\scriptscriptstyle\mathrm{op}}}

\newcommand{\sA}{\mathsf{A}}
\newcommand{\sB}{\mathsf{B}}
\newcommand{\sC}{\mathsf{C}}
\newcommand{\sD}{\mathsf{D}}
\newcommand{\sE}{\mathsf{E}}
\newcommand{\sM}{\mathsf{M}}
\newcommand{\sN}{\mathsf{N}}
\newcommand{\sH}{\mathsf{H}}

\newcommand{\sG}{\mathsf{G}}
\newcommand{\sU}{\mathsf{U}}
\newcommand{\sV}{\mathsf{V}}
\newcommand{\sW}{\mathsf{W}}
\newcommand{\sX}{\mathsf{X}}
\newcommand{\sk}{\mathsf{k}}

\newcommand{\fg}{\mathfrak{g}}

\newcommand\bC{\mathbb C}

\newcommand\bG{\mathbb G}

\newcommand\bZ{\mathbb Z}

\newcommand\cC{\mathcal C}
\newcommand\cD{\mathcal D}

\newcommand\cH{\mathcal H}
\newcommand\cI{\mathcal I}

\newcommand\cM{\mathcal M}

\newcommand\cS{\mathcal S}

\newcommand\cZ{\mathcal Z}

\newcommand\G{\Gamma}

\DeclareMathOperator{\cz}{\mathcal{CZ}}
\DeclareMathOperator{\hz}{\mathcal{HZ}}
\DeclareMathOperator{\hc}{\mathcal{HC}}
\DeclareMathOperator{\hgc}{\mathcal{HGC}}

\DeclareMathOperator{\End}{End}

\numberwithin{equation}{section}

\begin{document}
\author{Alexandru Chirvasitu}
\address{Department of Mathematics, University of Washington, Seattle, WA, USA}
\email{chirva@uw.edu}
\author{Pawe{\l} Kasprzak}
\address{Department of Mathematical Methods in Physics, Faculty of Physics, University of Warsaw, Poland}
\email{pawel.kasprzak@fuw.edu.pl}

\thanks{PK was partially supported by the Harmonia NCN grant
2012/06/M/ST1/00169.}
\begin{abstract}
The notion of  Hopf  center and Hopf cocenter of a Hopf algebra    is investigated by the extension theory of Hopf algebras. We   prove  that  each of them yields an exact sequence of Hopf algebras.  Moreover the exact sequences are shown to satisfy the faithful (co)flatness condition.  Hopf center and cocenter  are computed for $\sU_q(\fg)$ and the Hopf algebra  $\textrm{Pol}(\GG_q)$, where $\GG_q$ is the Drinfeld-Jimbo quantization of a compact semisimple simply connected Lie group $\bG$ and $\fg$ is a simple complex Lie algebra.  
\end{abstract}
\title{On the Hopf (co)center of a Hopf algebra}
\maketitle

\section*{Introduction} 
A Hopf algebra  is a mathematical  object   possessing a rich inherent symmetry, expressed by the compatible  algebra and coalgebra structures  together with  the antipode, which    flips around these  structures to their opposite counterparts. Hopf algebras  were   discovered in the context of algebraic topology \cite{Hopf} being applied then in the theory of algebraic groups, combinatorics, computer science, Galois theory and knot theory. In the  quantum  group incarnation they appear in noncommutative geometry and free probability. The theory of Hopf algebras is  studied and developed  not only as an effective tool applied elsewhere but also  due to the richness of their abstract theory, see e.g. \cite{rad_book}. 

One of the major factor stimulating   the  development of the Hopf algebra  theory   is its strong link with  group theory. In this paper we shall walk  along the mathematical  path beginning at the very initial level of  group theory,  while  one discusses the notion of a group's center, and   proceeding to the world of Hopf algebras,  getting  to the intriguing  area  where the extension theory of Hopf algebras together with the concept of  faithful  flatness  dominates the landscape. In our journey we shall experience not only the conceptual entertainment but we shall also compute the classes of motivating and instructive examples. 

Let  $\mathsf{G}$ be a group. The center of $\mathsf{G}$  \[\cZ(\mathsf{G}) = \{s\in\mathsf{G}: \forall t\in \mathsf{G}\,\, st = ts\,\}\] is a   subgroup of $\mathsf{G}$ which is normal: for all $t\in \mathsf{G}$ and $s\in \cZ(\mathsf{G}) $ we have $ts t^{-1} =s\in \cZ(\mathsf{G})$. In particular the quotient space $\mathsf{G}/\cZ(\mathsf{G})$ is equipped with the group structure such that the quotient map $\pi:\mathsf{G}\to \mathsf{G}/\cZ(\mathsf{G})$ is a surjective group homomorphism. Then  $\ker\pi = \cZ(\mathsf{G})$ and we have the short exact sequence in the category of groups     \begin{equation}\label{centexseq} \{e\}\rightarrow \cZ(\mathsf{G})\rightarrow \mathsf{G}\rightarrow  \mathsf{G}/\cZ(\mathsf{G}) \rightarrow  \{e\}\end{equation}  

Let $\mathsf{k}$ be a field. A finite group  $\mathsf{G}$ yields a  pair of  Hopf algebra triples 
\begin{itemize}
\item the group algebras triple: $\mathsf{k}[\cZ(\mathsf{G})], \mathsf{k}[\mathsf{G}] , \mathsf{k}[ \mathsf{G}/\cZ(\mathsf{G})]$
\item the $\mathsf{k}$-valued functions triple   $\mathsf{k}(\mathsf{G}/\cZ(\mathsf{G})),  \mathsf{k}(\mathsf{G}),\mathsf{k}(\cZ(\mathsf{G})) $ 
\end{itemize}
which are expected to  be incorporated into the pair of Hopf algebras exact sequences. 
 The concept of  exact sequences in the category of Hopf algebras  was introduced in  \cite{AD}, and indeed applied in our case gives rise to exact sequences 
\begin{equation}\label{exsecl}\begin{split}\mathsf{k}&\rightarrow
 \mathsf{k}[\cZ(\mathsf{G})]\rightarrow \mathsf{k}[\mathsf{G}] \rightarrow \mathsf{k}[ \mathsf{G}/\cZ(\mathsf{G})]\rightarrow\mathsf{k}\\
 \mathsf{k}&\rightarrow
 \mathsf{k}(\mathsf{G}/\cZ(\mathsf{G}))\rightarrow \mathsf{k}(\mathsf{G}) \rightarrow \mathsf{k}(\cZ(\mathsf{G}))\rightarrow\mathsf{k}
 \end{split}
\end{equation}
Dropping the group context in our example  and emphasizing   the Hopf context, we would be  led to the idea of Hopf center and cocenter.  
 It was first  (partially) discussed in  \cite{AD1} where  an  arbitrary Hopf algebra $\sA$ was assigned with the  left and right counterparts of  \eqref{exsecl}
 \begin{equation} \begin{split}\mathsf{k}&\rightarrow
 \hz(\sA)\rightarrow \sA \\
   \sA &\rightarrow \hc(\sA)\rightarrow\mathsf{k}
 \end{split}
\end{equation} 
The main subject of our   paper concerns   the corresponding right and left missing parts of the exact sequences. We show   the  exact sequences
 \begin{equation}\label{exha}
 \begin{split}\mathsf{k}&\rightarrow \hz(\sA)\rightarrow \sA\rightarrow  \sB\rightarrow \mathsf{k}\\
 \mathsf{k}&\rightarrow \sC\rightarrow \sA\rightarrow  \hc(\sA)\rightarrow  \mathsf{k}
 \end{split}
 \end{equation}
 indeed exist,  providing also certain accurate and/or approximate description of the exact sequences ingredients.
Our   results treat  also the faithful  (co)flatness issue, i.e. we show that $\sA$ is always faithfully flat over $\hz(\sA)\hookrightarrow \sA$ and faithfully coflat over $\sA\twoheadrightarrow\hc(\sA)$. This makes the   exact sequences \eqref{exha}  particularly nice.  
 
Passing to the examples we compute (co)centers  of the quantized universal enveloping algebra  $\sU_q(\fg)$ and the quantized Hopf algebra  $\textrm{Pol}(\GG_q)$ of function on $\GG_q$, the Drinfeld-Jimbo quantization of a compact semisimple simply connected Lie group $\sG$ with Lie algebra $\fg$.

\subsection*{Acknowledgements} 

We would like to thank the anonymous referee for many valuable suggestions on improving the initial draft.

\section{Preliminaries}

Throughout, we work over an algebraically closed field $\mathsf{k}$.

Our main references for the general theory of Hopf algebras are \cite{rad_book, Sweedler}. For a Hopf algebra $\sA$ over $\mathsf{k}$ we denote $\Delta$, $\varepsilon$, $S$ respectively comultiplication counit and antipode of $\sA$. In order to distinguish coproducts  of different Hopf algebras we shall write $\Delta^\sA$ and similarly  $S^\sA$ and $\varepsilon^\sA$ for antipode and counit respectively.  The kernel of counit will be denoted $\sA^+$ and for $\sB\subset \sA$ we write $\sB^+ = \sB\cap \sA^+$. 
 The Sweedler notation 
$\Delta (x) = x_{(1)}\otimes x_{(2)}$, $ x\in\sA$  will be freely used when convenient.  The center of the algebra $\sA$ will be denoted $\cZ(\sA)$. 

The fact that we are working over an algebraically closed field ensures that the simple coalgebras are exactly those dual to matrix algebras $M_n(\mathsf{k})$; we refer to such an object as a {\it matrix coalgebra}. The {\it coradical} of a coalgebra (i.e. the sum of its simple subcoalgebras; see \cite[Section 9.0]{Sweedler}) is a direct sum of matrix subcoalgebras. In particular, a cosemisimple cocommutative coalgebra will automatically be a group coalgebra; this remark is dual to the fact that the only semisimple commutative algebras over $\mathsf{k}$ are the finite products of copies of $\mathsf{k}$.

The monoidal category of  right $\sA$-comodules is denoted by
$\cM^{\sA}$, while the full subcategory of finite dimensional
comodules will be denoted by $\cM^{\sA}_f$. Similarly, the category of
left  comodules and the full subcategory of finite dimensional left
comodules will be denoted by $\leftidx{^\sA}{\cM}$ and
$\leftidx{^\sA}{\cM_f}$ respectively. We will work mostly with right
comodules. 

Recall that a monoidal category is \define{left rigid} if for every
object $x$ there is an object $x^*$ (the \define{left dual} of $x$)
with morphisms $\mathrm{ev}:x^*\otimes x\to 1$ and $\mathrm{coev}:1\to
x\otimes x^*$ (where $1$ is the monoidal unit) so that both 
\begin{equation*}
  \begin{tikzpicture}[auto,baseline=(current  bounding  box.center)]
    \path[anchor=base] (0,0) node (1) {$x$} +(3,.3) node (2) {$x\otimes
      x^*\otimes x$} +(6,0) node (3) {$x$};
         \draw[->] (1) to [bend left=6] node[pos=.5]{$\mathrm{coev}\otimes \mathrm{id}$} (2);     
         \draw[->] (2) to [bend left=6] node[pos=.5]{$\mathrm{id}\otimes\mathrm{ev}$}  (3);     
         \draw[->] (1) to [bend right=10] node[pos=.5,swap] {$\mathrm{id}$} (3);
  \end{tikzpicture}
\end{equation*}
and 
\begin{equation*}
  \begin{tikzpicture}[auto,baseline=(current  bounding  box.center)]
    \path[anchor=base] (0,0) node (1) {$x^*$} +(3,.3) node (2) {$x^*\otimes
      x\otimes x^*$} +(6,0) node (3) {$x^*$};
         \draw[->] (1) to [bend left=6] node[pos=.5]{$\mathrm{id}\otimes\mathrm{coev}$} (2);     
         \draw[->] (2) to [bend left=6] node[pos=.5]{$\mathrm{ev}\otimes\mathrm{id}$}  (3);     
         \draw[->] (1) to [bend right=10] node[pos=.5,swap] {$\mathrm{id}$} (3);
  \end{tikzpicture}
\end{equation*}
commute. 

Similarly, a monoidal category is \define{right rigid} if for every
object $x$ there is an object ${}^*x$ (its \define{right dual}) with
morphisms $\mathrm{ev}:x\otimes{}^*x\to 1$ and $\mathrm{coev}:1\to
{}^*x\otimes x$ that make $x$ a left dual to ${}^*x$. 

For any Hopf algebra $\sA$ the category $\cM^\sA_f$ is left rigid: for
a finite comodule $V$ the antipode $S$ can be used to put the comodule
structure $f\mapsto f_{(0)}\otimes f_{(1)}$ such that
\begin{equation}\label{eq:1}
  f_{(0)}(v)f_{(1)} = f(v_{(0)})Sv_{(1)},\ \forall v\in \sV,\ f\in \sV^*,
\end{equation}
where $v\mapsto v_{(0)}\otimes v_{(1)}$ is our version of Sweedler
notation for right comodule structures. This comodule structure is
such that the usual evaluation $\sV^*\otimes \sV\to \mathsf{k}$ and
dual basis coevaluation map $\mathsf{k}\to \sV\otimes \sV^*$ are comodule
morphisms. 

We can also make $\cM^\sA_f$ into a \define{right} rigid monoidal
category, but we need $\sA$ to have bijective antipode: the procedure
is the same as above, except that the inverse of the antipode is used
instead of $S$ in \Cref{eq:1}.

The adjoint action $\ad:\sA\to\End(\sA)$ of a Hopf algebra on itself is $\ad_xy = x_{(1)}yS(x_{(2)})$. 
The map $\ad:\sA\rightarrow \End(\sA)$ is an algebra homomorphism
\[\ad_{xy} = \ad_{x}\ad_{y}\] satisfying the Leibniz rule 
\[\ad_x(yz) = \ad_{x_{(1)}}(y)\ad_{x_{(2)}}(z)\]
and 
\begin{equation}\label{basic}
\Delta (\ad_x(y))  = x_{(1)}y_{(1)}S(x_{(4)})\otimes \ad_{x_{(2)}}(y_{(2)})
\end{equation}
It is known   that 
\begin{equation}\label{propo1} 
\cZ(\sA) = \{x\in\sA:\ad_y(x) = \varepsilon(y)x \textrm{ for all } y\in \sA\}
\end{equation}
For the proofs of the claims in the present paragraph see e.g. \cite[Section 1.3]{Joseph}. 

The adjoint coaction $\cad:\sA\rightarrow \sA\otimes\sA$ is  
\begin{equation}\label{aldef}
\cad(x) = x_{(2)}\otimes S(x_{(1)})x_{(3)}
\end{equation} 
Then $\cad$ is a linear map satisfying \cite[2.3.1-2.3.3]{AD1}. 
Moreover
\begin{equation}\label{coact} (\cad\otimes\id)(\cad(x)) = x_{(3)}\otimes S(x_{(2)})x_{(4)}\otimes S(x_{(1)})x_{(5)} = (\id\otimes\Delta)(\cad(x))\end{equation} for all $x\in\sA$.

We recall the definition of exact sequence of Hopf algebras following \cite{AD}. 
\begin{definition}\label{esdef}
Consider a sequence of morphisms of Hopf algebras 
\begin{equation}\label{seqha}\sk\rightarrow\sA\xrightarrow{\iota}\sC\xrightarrow{\pi}\sB\rightarrow\sk\end{equation} 
We say that \eqref{seqha} is exact if 
\begin{enumerate}
\item $\iota$ is injective;
\item $\pi$ is surjective;
\item $\ker\pi = \sC\iota(\sA)^+$;
\item $\iota(\sA) = \{x\in\sC:(\pi\otimes\id)\Delta(x) = \I\otimes x\}$.
\end{enumerate}
\end{definition}
In the sequel we will make frequent use of the notions of faithfully (co)flat morphisms of Hopf algebras. For completeness we give the following definition.

\begin{definition}
A left module $\sV$ over an algebra $\sA$ is \define{flat} if the functor $-\otimes_\sA \sV$ on right modules preserves monomorphisms (i.e. injections). $\sV$ is \define{faithfully flat} if it is flat and $-\otimes_\sA \sV$ is also faithful. Right-handed notions are defined analogously for right modules. 

A morphism $\sA\to \sB$ of algebras is left (right) \define{(faithfully) flat} if $\sB$ is (faithfully) flat as a left (right) $\sA$-module. 

A right comodule $\sV$ over a coalgebra $\sC$ is \define{coflat} if the cotensor product functor $\sV\square_\sC-:{}^\sC\cM\to\mathrm{Vect}$ (defined dually to the tensor product; see \cite[$\S$ 10.1]{BrWi}) preserves epimorphisms (i.e. surjections). $\sV$ is \define{faithfully coflat} if it  is coflat and $\sV\square_\sC-$ is also faithful. Once more, there are analogous left-handed notions for left comodules. 

A coalgebra morphism $\sC\to \sD$ is left (right) (faithfully) coflat if $\sC$ is (faithfully) coflat as a left (right) $\sD$-comodule. 
\end{definition}

\begin{remark}\label{re.coflat=inj}
  It is shown in \cite{tak_form} that a $\sD$-comodule is faithfully coflat if and only if it is an \define{injective cogenerator}, in the sense that it is injective and all indecomposable injective $\sC$-comodules appear as summands in $\sD$.

In particular, if $\sD$ is cosemisimple and $\sC\to \sD$ is a surjective morphism of coalgebras then it is (left and right) faithfully coflat. Indeed, our assumptions ensure that all $\sD$-comodules are injective and the surjection $\sC\to \sD$ splits. The conclusion then follows from the fact that $\sD$ itself is an injective cogenerator in both $\cM^\sD$ and ${}^\sD\cM$ (this observation is dual to the fact that an algebra is a projective generator in its category of either left or right modules). 
\end{remark}

\begin{remark}\label{remex}Suppose that the antipode   $S^{\sC}$  of $\sC$ is bijective.
  Let $\sA\xrightarrow{\iota}\sC$ be a faithfully flat morphism of Hopf algebras and suppose that $\iota$ is ad-invariant, i.e. $\ad_x(\iota(\sA))\subset \iota(\sA)$ for all $x\in\sC$. Then conditions (1),(2),(3) imply condition (4) of \Cref{esdef}.
Conversely let $\sC\xrightarrow{\pi}\sB$ be a faithfully coflat morphism of Hopf algebras  and suppose that $\pi$ is normal in the sense of \cite[Definition 1.1.5]{AD}, i.e. \[(\pi\otimes\id)\Delta(x) = \I\otimes x  \iff    (\id\otimes\pi)\Delta(x) = x\otimes \I\] for all $x\in\sC$. 
Then conditions (1),(2),(4) imply condition (3) of \Cref{esdef}.
\end{remark}

The leg numbering notation will be used throughout the paper, i.e. for $x\in \sA\otimes  \sA$  we define $x_{12},x_{13}, x_{23}\in \sA\otimes  \sA\otimes  \sA$, where e.g. $x_{23} = \I\otimes x$. We shall also write $x_{21}$ for the flipped version of $x_{12}$: if $x = a\otimes b$ then  $x_{21} = b\otimes a\otimes \I$ and  similarly for $x_{32}, x_{31}$. Leg numbering notation is also used  for linear maps. If    $T: \sA\otimes  \sA\to \sA\otimes \sA$  then  $\id\otimes T:  \sA\otimes  \sA\otimes  \sA\to  \sA\otimes  \sA\otimes  \sA$ will be denoted by  $T_{23}$, etc. We warn the reader not to confuse the Sweedler notation where the bracketed indices are used, with the unbracketed leg numbering notation.  
 
\section{Hopf center}
Let us begin with the following definition. 
\begin{definition} \cite[Definition 2.2.3]{AD1}
Let $\sA$ be a Hopf algebra.
The Hopf  center $\hz(\sA)$ of $\sA$ is   the largest Hopf subalgebra of $\sA$ contained in the center $\cZ(\sA)$.
\end{definition}
 Analyzing the reasoning which proceeds \cite[Definition 2.2.3]{AD1} we see that if $S$ is bijective then we have
\begin{equation}\label{defH}\hz(\sA) = \{x\in \sA:(\id\otimes\Delta)(\Delta(x)) \in \sA\otimes \cZ(\sA)\otimes \sA\}\end{equation}

In what follows we shall give a slightly weaker condition for $x\in \sA$ to belong to 
 $\hz(\sA)$. In order to see the relation with \cite[Section 2]{AD1} let us note that applying $\varepsilon$ to the third leg of  $(\id\otimes\Delta)(\Delta(x))\in \sA\otimes \cZ(\sA)\otimes \sA $ we get  
 $\Delta(x) \in  \sA\otimes \cZ(\sA)$. Let us emphasize that the  invertibility assumption which enters the formulation of the next theorem can  be   dropped, by using the Tannakian description of the Hopf center which we shall give later.  For a locally compact quantum group version of the following result see \cite{KSS}.
 
\begin{theorem} \label{main}
Let $\sA$ be a Hopf algebra with bijective antipode.
 Let us define 
 \[\sM =\{x\in \sA: \Delta(x) \in \sA\otimes  \cZ(\sA)\}\]
 Then  $\sM= \hz(\sA)$.
\end{theorem}
\begin{proof}
Let us first note that if $x\in \sM$ then $x = (\varepsilon\otimes\id)(\Delta(x))\in \cZ(\sA)$ thus $\sM \subset \cZ(\sA)$. 
 Let us show that 
 \begin{equation}\label{zatenza}\Delta(\sM)\subset  \cZ(\sA)\otimes \cZ(\sA)\end{equation}  Using \Cref{propo1} it is enough to show that for  $y\in \sM$ we have 
 \[(\ad_x\otimes\id)(\Delta(y)) = \varepsilon(x)\Delta(y)\]  for all $x\in \sA$. 
 We compute 
 \[
\begin{split}
(\ad_x\otimes \id)(\Delta(y)) & = x_{(1)}y_{(1)}S(x_{(2)})\otimes y_{(2)}\\
& = x_{(1)}y_{(1)}S(x_{(3)})\otimes \varepsilon(x_{(2)}) y_{(2)}\\
& = x_{(1)}y_{(1)}S(x_{(4)})\otimes x_{(2)}y_{(2)}S(x_{(3)})\\
& =  \Delta(\ad_x(y))\\
& =  \Delta(\varepsilon (x)(y))\\
& =  \varepsilon (x) \Delta(y)\\
 \end{split}
\]
where in  the third equality we used that $y_{(2)} \in \cZ(\sA)$; in the fourth equality we used \eqref{basic}; in the fifth equality we used $y\in \cZ(\sA)$. 

Noting that for  $x\in \sM$  
\[(\id\otimes \Delta)(\Delta(x)) = (\Delta\otimes\id)(\Delta(x)) \in \sA\otimes \sA\otimes \cZ(\sA)\] we get 
\begin{equation}\label{coidM}\Delta(\sM)\subset \sA\otimes \sM\end{equation}

The invertibility of $S$ yields $S(\cZ(\sA))
= \cZ(\sA)$ thus for all $x\in \sM$ we get 
\[\Delta(S(x)) = (S\otimes S)(\Delta^{\textrm{op}}(x)) \in S(\cZ(\sA))\otimes  S(\cZ(\sA)) =\cZ(\sA)\otimes\cZ(\sA) \]
Using \Cref{zatenza} and the bijectivity of $S$ we conclude that
 $S(\sM) =  \sM$.  The $S$-invariance of $\sM$  together with \Cref{coidM} implies that  $\sM$ is a Hopf subalgebra. 

If $\sN\subset \cZ(\sA)$ is a Hopf subalgebra then
\[\Delta(\sN)\subset \sA\otimes  \cZ(\sA)\]
thus $\sN\subset \sM$. This means that $M$ is the largest Hopf subalgebra of $\sA$, and hence coincides with $\hz(\sA)$.  
\end{proof}
 
\begin{lemma}\label{easy}
Let $\sA$ be a Hopf algebra with bijective antipode.
Then the center  of $\sA$ equals
\[\hz(\sA) = \{x\in\cZ(\sA): \Delta(x)\in\cZ(\sA)\otimes \cZ(\sA)\}\]
\end{lemma}
\begin{proof}
Clearly if $x\in\cZ(\sA)$ and $\Delta(x)\in\cZ(\sA)\otimes \cZ(\sA)$ then using \Cref{main} we get  $x\in\hz(\sA) $. Conversely if $x\in \hz(\sA)$ then   $x\in  \cZ(\sA)$ and $\Delta(x)\in\cZ(\sA)\otimes \cZ(\sA)$. 
\end{proof}
\begin{lemma}\label{easy1}
Let $\sA$ be a Hopf algebra with bijective antipode and $\sV$  a right $\sA$-comodule with the corresponding   map $\rho:\sV\rightarrow \sV\otimes \sA$ satisfying 
\[\rho(\sV)\subset \sV\otimes\cZ(\sA)\] Then 
\[\rho(\sV)\subset\sV\otimes\hz(\sA)\] In particular  $\sV$ can be viewed as a right $\hz(\sA)$-comodule. The same holds for  left comodules. 
\end{lemma}
\begin{proof}
Let us note that 
\[
\begin{split}
(\id\otimes\Delta)(\rho(\sV)) &= (\rho\otimes\id)(\rho(\sV))
\\&\subset (\rho\otimes\id)(\sV\otimes \cZ(\sA))\\
&\subset \sV\otimes\cZ(\sA)\otimes \cZ(\sA)
\end{split}
\]
Using  \Cref{easy} we conclude that 
$\rho(\sV)\subset \sV\otimes \hz(\sA)$ and the rest is clear. 
\end{proof}
Let us note we have $\ad_x(\hz(\sA)) = \varepsilon(x)\hz(\sA)$, thus   the embedding $\iota:\hz(\sA)\to \sA$ is ad-invariant. Since $\hz(\sA)$ is commutative, we can use  \cite[Proposition 3.12]{ArkGait} to conclude that $\iota$ is a faithfuly flat morphism. Using  \Cref{remex} we get the following theorem. 
\begin{theorem} 
Let $\sA$ be a Hopf algebra with bijective antipode. Then the embedding $\iota:\hz(\sA)\to \sA$ is faithfully flat and there exists   a Hopf algebra $\sB $ and a surjective morphism  $\pi:\sA\rightarrow  \sB $ of Hopf algebras such that the following sequence of Hopf algebras is exact
  \[\sk\rightarrow \hz(\sA)\rightarrow \sA\rightarrow  \sB \rightarrow  \sk\] 
\end{theorem}

\subsection{A Tannakian description of the Hopf center}\label{subse.tann}

According to the general philosophy of Tannaka reconstruction (as presented e.g. in \cite{schau_tann}), the inclusion $\hz(\sA)\subseteq \sA$ of Hopf algebras can be described completely by identifying the full inclusion $\cM^{\hz(\sA)}_f\to \cM^{\sA}_f$ of categories of finite-dimensional comodules. We refer to loc. cit. for background on Tannaka reconstruction for coalgebras, bialgebras and Hopf algebras. 

The following discussion applies to bialgebras, so throughout \Cref{subse.tann} $\sA$ will be a fixed but arbitrary bialgebra; when we require it to be a Hopf algebra (or Hopf algebra with bijective antipode) we will say so explicitly.

\begin{definition}\label{def.central_comod}
  An $\sA$-comodule $\sV\in \cM^\sA_f$ is \define{central} if the flip map
\begin{equation*}
    \tau:\sV\otimes \sW\to \sW\otimes \sV,\quad v\otimes w\mapsto w\otimes v
  \end{equation*}
is a comodule morphism for all $\sW\in \cM^\sA_f$. 
\end{definition}

Consider the full subcategory $\cC\subseteq \cM^\sA_f$ consisting of the
central comodules. Before stating the next result, recall also the following definition from \cite{schau_tann} (see Definition 2.2.10 therein); here, a \define{subquotient} of a comodule is a quotient of a subcomodule (or equivalently, a subcomodule of a quotient comodule).

\begin{definition}\label{def.closed}
  A full subcategory $\cC$ of an abelian category $\cD$ is \define{closed} if all subquotients in $\cD$ of an object $x\in \cC$ are again in $\cC$. 
\end{definition}

\begin{lemma}
  \begin{enumerate}
     \item[(a)]   $\cC\subseteq \cM^\sA_f$ is a closed monoidal
       subcategory.    
     \item[(b)] If $\sA$ is a Hopf algebra, then $\cC$ is left rigid. 
     \item[(c)] If in addition the antipode of $\sA$ is bijective, then
       $\cC$ is both left and right rigid. 
  \end{enumerate}
\end{lemma}
\begin{proof}
  We leave part (a) to the reader, and instead focus on part (b). In
other words, we will argue that if $\cM^\sA_f$ admits left duals
$\sV\mapsto \sV^*$, then $\cC$ is closed under taking such duals. The
argument for right duals in part (c) is completely analogous, so we
will omit that as well. 

To check (b), note first that for any $\sV\in \cC$ the left dual $\sV^*$ and the right dual ${}^*\sV$ are isomorphic through the canonical map.  Now, the flip induces a natural isomorphism $\sV\otimes -\cong -\otimes \sV$ of endofunctors of $\cM^{\sA}_f$, and hence also induces a natural isomorphism between their right adjoints, 
\begin{equation*}
  {}^*\sV\otimes -\cong \sV^*\otimes - \text{ and } -\otimes \sV^*
\end{equation*}
respectively. 
\end{proof}

According to \cite[Lemma 2.2.12]{schau_tann} (or rather a monoidal version thereof), this implies that $\cC$ is of the form $\cM^\sC_f$ for some sub-bialgebra $\sC\subseteq \sA$. We will see below that in the case of Hopf algebras $\sC$ is precisely the Hopf center discussed above. First, we need

\begin{lemma}\label{le.center_tann_aux}
 Let  $\sV\in \cM^{\sA}_f$. Then  $\sV\in \cC\subset\cM^{\sA}_f$ if and only if the coefficient coalgebra $ C (\sV)\subset \sA$ of $\sV$ is contained in the center. 
\end{lemma}
\begin{proof}
  If the property from \Cref{def.central_comod} holds for a specific comodule $\sW$ we will say that $\sV$ \define{commutes} with $W$. 

Let $v_i$ and $w_k$ be bases for $\sV$ and $\sW$ respectively, and denote their respective comodule structures by
\begin{equation*}
  v_j\mapsto v_i\otimes a_{ij} \text{ and } w_\ell\mapsto w_k\otimes b_{k\ell},
\end{equation*}
where repeated indices indicate summation. Then, $\sV$ commutes with $\sW$ if and only if the diagram 
\begin{equation}\label{eq.center_tann_aux}
  \begin{tikzpicture}[auto,baseline=(current  bounding  box.center)]
    \path[anchor=base] (0,0) node (1) {$\sV\otimes \sW$} +(4,0) node (2) {$\sV\otimes \sW\otimes \sA$} +(0,-2) node (3) {$\sW\otimes \sV$} +(4,-2) node (4) {$\sW\otimes \sV\otimes \sA$};
         \draw[->] (1) to  (2);     
         \draw[->] (3) to  (4);     
         \draw[->] (1) to node[pos=.5,swap] {$\tau$}  (3);
         \draw[->] (2) to node[pos=.5] {$\tau\otimes\id_\sA$}  (4);
  \end{tikzpicture}
\end{equation}  
is commutative (the horizontal arrows being the comodule structure maps of the two tensor products). 

Applying the two paths (right-down and down-right) in \Cref{eq.center_tann_aux} to the element $v_j\otimes w_\ell\in \sV\otimes \sW$ and then reading off the coefficient of $w_k\otimes v_i\in \sW\otimes \sV$, we see that this is equivalent to $a_{ij}b_{k\ell} = b_{k\ell}a_{ij}$ for all possible index choices. In other words, all elements of the coefficient coalgebra $C(\sV)$ must commute with all elements of $C(\sW)$. The conclusion follows from the fact that as $\sW$ is allowed to range over all finite-dimensional comodules, the resulting $C(\sW)$ jointly span all of $\sA$. 
\end{proof}

Consider now the  sub-bialgebra $\sC$ which has $\cC\subseteq \cM^\sA_f$ as its comodule category. It is the linear span of the coefficients
of its comodules, so it must be contained in the center $\cZ(\sA)$. On
the other hand, any sub-bialgebra of $\sA$ contained in $\cZ(\sA)$ corresponds
via Tannaka reconstruction to some closed monoidal subcategory of
$\cM^\sA_f$ consisting of central comodules, which must then be
contained in $\cC$. In conclusion, denoting again by $\hz(\sA)$ the largest
sub-bialgebra of $\sA$ contained in the center $\cZ(\sA)$ (for
consistency, and also because we are mostly interested in Hopf algebras
anyway) we get

\begin{theorem}\label{th.center_tann}
The inclusion $\cC\to \cM^\sA_f$ of the category of central comodules is
equivalent to $\cM^{\hz}_f\subseteq \cM^\sA_f$, where $\hz=\hz(A)$ is
the largest sub-bialgebra of $\sA$ contained in the center $\cZ$ of $\sA$. 

Moreover, if $\sA$ is a Hopf algebra (or Hopf algebra with bijective
antipode), then $\hz$ is also the largest Hopf subalgebra
(respectively Hopf subalgebra with bijective antipode) contained in $\cZ$
\qedhere  
\end{theorem}

Finally, we have the following alternative characterization of the Hopf subalgebra $\hz\subseteq \sA$ from \Cref{th.center_tann}.

\begin{theorem}\label{th.alt_char}
Let $\sA$ be an arbitrary Hopf algebra. The largest Hopf subalgebra $\hz\subseteq \sA$ contained in the center $\cZ$ of $\sA$ coincides with the set
  \begin{equation*}
    \sM =\{x\in \sA: \Delta(x) \in \cZ\otimes \sA\} 
  \end{equation*}
\end{theorem}
\begin{proof}
  Clearly, $\hz$ is contained in $\sM$; the interesting inclusion is the
  opposite one. 

Let $x\in \sM$. Since 
\begin{equation*}
  ((\Delta\otimes\id)\circ\Delta)(x) =
  ((\id\otimes\Delta)\circ\Delta)(x) \in \cZ\otimes \sA\otimes \sA, 
\end{equation*}
we have $\Delta(x)\in \sM\otimes \sA$, and hence $\sM$ can be regarded as a
right $\sA$-comodule. We can moreover check that $\Delta(\sM)\subseteq
\cZ\otimes \cZ$ just as in the proof of \Cref{main}, and hence the
coefficients of $\sM$, as a right $\sA$-comodule, lie in $\cZ$. Let us emphasize that at this stage of the proof of \Cref{main} we do not use invertibility of $S$.  

It now follows from \Cref{le.center_tann_aux} that $\sM$ is central and hence must be an
$\hz$-comodule. Finally, the conclusion follows by applying
$\varepsilon\otimes\id$ to $\Delta(x)\in \sM\otimes\hz$ for $x\in \sM$. 
\end{proof}

\begin{remark}
  We could just as easily carried out the discussion for left
  comodules. Since the characterization of $\hz$ as the largest
  sub-bialgebra contained in the center is blind to this distinction,
  we could have substituted $A\otimes \cZ$ for $\cZ\otimes A$ in
  \Cref{th.alt_char}. Thus, we can avoid antipode bijectivity in \Cref{main}. 
\end{remark}


\section{Hopf cocenter} 
In \cite[Section 2]{AD1} the following notion is considered:
\begin{definition}\label{def.hopf_cocentral_morphism} 
Let $\pi:\sA\rightarrow \sB $ be a morphism of Hopf algebras. We say that $\pi$ is \define{cocentral} if \[(\id\otimes\pi)\circ\Delta = (\id\otimes\pi)\circ\Delta^\op\]
\end{definition}
 
In \cite[Lemma 2.3.7]{AD1} it was proved that there exists a  quotient Hopf algebra  $\hc(\sA)$ of $\sA$ such that all cocentral morphisms from $\sA$ factor uniquely through $\hc(\sA)$; moreover, $\hc(\sA)$ is automatically cocommutative. Strictly speaking, in loc. cit. Hopf algebras are always assumed to have bijective antipode. Nevertheless, the existence of $\hc(\sA)$ for arbitrary Hopf algebras follows from its existence for Hopf algebras with bijective antipode. 

Indeed, note that a cocentral morphism $\sA\to \sB$ factors through the universal involutive quotient $\overline{\sA} = \sA/(S^2-\id)$ of $\sA$. This means that the quotient $\sA\to \overline{\sA}\to \hc(\overline{\sA})$ has the universal property defining $\hc(A)$.   

With this in place, we can now give

\begin{definition}
  The \define{cocenter} $\hc(\sA)$ of a Hopf algebra $\sA$ is the quotient Hopf algebra $\sA\to\hc(\sA)$ that is universal among cocentral quotients of $\sA$. 
\end{definition}
A similar concept for cosemisimple Hopf algebras with the assumption that $\hc(\sA)$ is cosemisimple was introduced in \cite{chi_cos}. The following result shows that the two coincide when they both make sense. 

\begin{theorem}\label{th.an=ch}
If $\sA$ is a cosemisimple Hopf algebra, then its cocenter $\hc(\sA)$ is a group algebra.   
\end{theorem}
\begin{proof}
  According to the proof of \cite[Proposition 2.4]{chi_coc}, a Hopf algebra map $\pi:\sA\to  \sB $ is cocentral if and only if for any functional $\varphi\in  \sB ^*$ the action of $\varphi$ induced by $\pi$ on any (right, say) $\sA$-comodule is an endomorphism in $\cM^{\sA}$. 

We can now continue as in the proof in loc. cit. When $\sA$ is cosemisimple the above condition is equivalent to saying that for every simple $\sA$-comodule $V$, the action
\begin{equation*}
  \triangleright:\sB^*\otimes \sV\to \sV
\end{equation*}
of $ \sB ^*$ on $\sV$ is of the form
\begin{equation}\label{eq.gplike}
  \varphi\triangleright v = \varphi(g)v,\quad \forall \varphi\in  \sB ^*,\ v\in V
\end{equation}
for some grouplike element $g\in  \sB $. 

If furthermore $\sA\to  \sB $ is surjective (which is the case for $ \sB =\hc(\sA)$), then $ \sB $ will be generated by the grouplikes from \Cref{eq.gplike}. This proves the claim. 
\end{proof}

In \cite{chi_coc} an interpretation was given in the cosemisimple case for the group whose group algebra $\hc(\sA)$ is. In order to adapt that discussion to the present, general setting, we need the following notions. First, we recall from \cite[$\S$ 5.1]{ComOst11} the concept of block in an abelian category (applied in our case to categories of comodules).

\begin{definition}\label{def.blocks}
  The \define{blocks} of an abelian category are the classes of the weakest equivalence relation on the set of indecomposable objects generated by the requiring that two objects are equivalent if there are non-zero morphisms between them. 
\end{definition}

\begin{definition}\label{def.univ_grading}
The \define{universal grading group} $\G(\sA)$ of a Hopf algebra $\sA$ is the group with one generator $g_\sV$ for each simple  $\sV\in\cM^{\sA}$  subject to the following relations 
\begin{itemize}
   \item $g_\sV$ only depends on the block of the category $\cM^{\sA}$ of $\sA$-comodules that $\sV$ belongs to;
   \item $g_\sk=1$, where $\sk\in \cM^{\sA}$ denotes the trivial comodule;
   \item For simple $\sA$-comodules $\sV$ and $\sW$, the product $g_\sV g_\sW$ equals $g_\sX$ for any simple subquotient $\sX$ of $\sV\otimes \sW$. 
\end{itemize}
\end{definition}
\begin{remark}
  Note that it follows from the definition that
  $g_{\sV^*}=g_\sV^{-1}$, and similarly for the right dual ${}^*\sV$
  when the antipode is bijective (the bijectivity of the antipode is
  necessary in order to define a coaction on ${}^*\sV$ that makes it
  into a right dual).
\end{remark}

\begin{remark}
  The universal grading group is the {\it chain group} of \cite{BL,Mu}.
\end{remark}

We always have a Hopf algebra map $\sA\to \sk\G(\sA)$ from a Hopf algebra to the group algebra of its universal grading group: at the level of comodule categories, the corresponding functor $\cM^{\sA}\to \cM^{\sk\G(\sA)}$ sends an indecomposable $\sA$-comodule $\sV$ to the direct sum of $\dim(\sV)$ copies of the $\sk\G(\sA)$-comodule $\sk g$, where $g$ is the element of $\G(\sA)$ corresponding to the simple socle of $\sV$ (recall that the {\it socle} of a comodule is the sum of its simple subcomodules). 

When $\sA$ is cosemisimple $\sA\to \sk\G(\sA)$ can be identified with $\sA\to \hc(\sA)$ (see \cite{chi_coc}). In general, we again get a cocentral Hopf algebra quotient, smaller than $\hc(\sA)$.

\begin{definition}
  The \define{group cocenter} of $\sA$ is the quotient $\sA\to \sk\G(\sA)$ from the above discussion. 
We will sometimes also write $\hgc(\sA)$ for $\sk\G(\sA)$.  
\end{definition}

The following result shows that the group cocenter always fits into an exact sequence. 

\begin{theorem}\label{hgc_sequence}
  The group cocenter $\hgc$ of $\sA$ fits into an exact sequence 
  \begin{equation*}
    k\to  \sC \to \sA\to \hgc\to k
  \end{equation*}
of Hopf algebras, with $\sA$ faithfully flat over $ \sC $ and faithfully coflat over $\hgc$. 
\end{theorem}
\begin{proof}
 Denote the surjection $\sA\to \hgc$ by $\pi$ and let
 \begin{equation*}
    \sC =\{x\in \sA\ |\ (\pi\otimes\id)(\Delta(x))=1\otimes x\}.
 \end{equation*}
 Since $\pi$ is central, $\sC$ is a Hopf algebra. 
Using  \cite[Theorem 2]{tak}, we obtain  faithful flatness of $\sA$ over $ \sC $ given faithful coflatness of $\sA$ over $\hgc$; this latter condition holds simply because $\pi$ is a surjection onto a cosemisimple Hopf algebra, and hence automatically faithfully coflat (see \Cref{re.coflat=inj}). 
\end{proof}

Now let $\sA$ be a Hopf algebra  with bijective antipode  and denote by $\hc$ and $\hgc$ its Hopf cocenter and group cocenter respectively. By the universality of $\sA\to \hc$, the surjection $\sA\to \hgc$ factors as
\begin{equation}\label{eq.factorize}
  \begin{tikzpicture}[auto,baseline=(current  bounding  box.center)]
    \path[anchor=base] (0,0) node (1) {$\sA$} +(4,0) node (2) {$\hgc$} +(2,-1) node (3) {$\hc$};
         \draw[->] (1) to  (2);     
         \draw[->] (1) to[bend right=10]  node[left,pos=.55] {}  (3);
         \draw[->] (3) to[bend right=10]  node[right,pos=.2,swap] {$\phantom{i}\pi$}  (2);
  \end{tikzpicture}
\end{equation}  
for a Hopf algebra surjection $\pi$. Note that $\pi$ restricts to a surjection of the group $\sG(\hc)$ of grouplikes in $\hc$ onto $\G=\G(\sA)=\sG(\hgc)$ (e.g. by \cite[Proposition 4.1.7]{rad_book}). The next result says, essentially, that this surjection is in fact also one-to-one.

\begin{proposition}\label{pr.cocentral_gplikes}
  The surjection $\pi$ from \Cref{eq.factorize} is split as in the diagram
\begin{equation}\label{eq.split}
  \begin{tikzpicture}[auto,baseline=(current  bounding  box.center)]
    \path[anchor=base] (0,0) node (1) {$\hc$} +(3,0) node (2) {$\hgc$};
         \draw[->] (1) to[bend left=20] node[pos=.5] {$\pi$}  (2);     
         \draw[->] (2) to[bend left=20]  node[pos=.5] {$\iota$}  (1);
  \end{tikzpicture}
\end{equation}
 by an inclusion $\iota$ of Hopf algebras.  
\end{proposition}
\begin{proof}
  All we have to show is that, as in the discussion preceding the statement, $\pi$ induces an isomorphism at the level of grouplikes. In other words, we need to prove that the grouplikes of $\hc$ satisfy the relations listed in \Cref{def.univ_grading}.

As in the proof of \Cref{th.an=ch}, every grouplike element $g\in \sG(\hc)$ induces an action of $\hc^*$ on every $\sA$-comodule via the character
\begin{equation*}
  \varphi\mapsto \varphi(g): \hc^*\to \sk. 
\end{equation*}
This character must be the same for comodules belonging to the same block because $\hc^*$ acts by $\sA$-comodule maps, hence the first bullet point in \Cref{def.univ_grading}.

The second bullet point in that definition clearly holds for the group $\sG(\hc)$. 

Finally, the third bullet point of \Cref{def.univ_grading} holds because $\hc$ is a Hopf algebra coacting on irreducible $\sA$-comodules via the grouplikes $\sG(\hc)$, and tensor products of comodules correspond to products of coalgebra coefficients with respect to the $\hc$ coaction.  
\end{proof}


\section{ Hopf (co)center and   adjoint coaction} 
Let $\sA$ be a Hopf algebra and $\cad:\sA\rightarrow \sA\otimes \sA$ the adjoint coaction  as defined in \Cref{aldef}.  In this section we shall describe the links between 
\begin{itemize}
\item the adjoint coaction $\cad$ and the Hopf center $\hz(\sA)$,
\item the adjoint coaction $\cad$ and the Hopf cocenter $\hc(\sA)$.
\end{itemize} 
 In what follows we shall provide necessary and sufficient conditions for $\cad:\sA\to \sA\otimes \sA$ to be an algebra homomorphism. Clearly, this   is the case if $\sA$ is commutative. 
\begin{theorem}Let $\sA$ be a Hopf algebra and $\cad:\sA\rightarrow \sA\otimes \sA$ the adjoint coaction  \Cref{aldef}.
Then  $\cad$ is an algebra homomorphism if and only if 
\[\cad(\sA)\subset \sA\otimes \cZ(\sA)\] 
\end{theorem}
\begin{proof} The map $\cad$ is an algebra homomorphism if and only if its flipped version $\delta:\sA\to \sA\otimes \sA$ is where 
 $\delta(x) = S(x_{(1)})x_{(3)}\otimes x_{(2)} $.
In the course of the proof we shall use notation of the proof of \cite[Theorem 3.3]{BudzKasp}. In particular we define the invertible linear map $W:\sA\otimes\sA\to\sA\otimes\sA$:
\[W  (a\otimes a') =  a_{(1)}\otimes a_{(2)}a'\]   
It is easy to check that 
 \begin{equation}\label{pent}W_{23}W_{12}W_{23}^{-1} = W_{12}W_{13}\end{equation}
Let  us consider a linear map  $U:\sA \otimes \sA\otimes \sA \to\sA \otimes \sA\otimes \sA$  
\[U(a\otimes b\otimes c) = (\id\otimes\delta)(\Delta (a))(\I\otimes b\otimes c)=a_{(1)}\otimes S(a_{(2)})a_{(4)}b\otimes a_{(3)}c\]
for all $a,b,c\in \sA$. 
Then    $U$  
satisfies   
\begin{equation}\label{simeq}W_{12}^{-1}U_{234} W_{12}  = U_{134} U_{234}\end{equation}
For the proof of \eqref{simeq}  see the proof of \cite[Theorem 3.3]{BudzKasp}. Note that this is the place where the assumption that $\cad$ is an algebra homomorphism is used. 

Let us note that 
\begin{equation}\label{Uform}
U = W_{12}^{-1}W_{13}W_{12}
\end{equation} Indeed 
\[\begin{split}
W_{12}^{-1}W_{13}W_{12}(a\otimes b\otimes c)&=W_{12}^{-1}W_{13}(a_{(1)}\otimes a_{(2)}b\otimes c)\\
& = W_{12}^{-1}(a_{(1)}\otimes a_{(3)}b\otimes a_{(2)}c)\\
&=a_{(1)}\otimes S(a_{(2)})a_{(4)}b\otimes a_{(3)}c
\end{split}
\]
Now, using \Cref{Uform} and \Cref{pent} we get 
\[\begin{split} W_{12}^{-1}U_{234} W_{12}&= W_{12}^{-1}W_{23}^{-1}W_{24}W_{23}W_{12} \\&=
W_{23}^{-1}W_{13}^{-1}W_{14}W_{24}W_{13}W_{23}\\&=
W_{23}^{-1}W_{13}^{-1}W_{14}W_{13}W_{24}W_{23}
\end{split}
\]
 On the other hand, again using \Cref{Uform} we get 
\[U_{134} U_{234} = W_{13}^{-1}W_{14}W_{13}W_{23}^{-1}W_{24}W_{23}\] 
Comparing both expressions we conclude that 
\[ W_{13}^{-1}W_{14}W_{13}W_{23}^{-1}=W_{23}^{-1}W_{13}^{-1}W_{14}W_{13}\] 
or equivalently
\begin{equation}\label{eqWU}W_{23}U_{134} = U_{134}W_{23} \end{equation}
Applying both sides of \eqref{eqWU} to $a\otimes b\otimes \I\otimes \I\otimes \I$ we get 
\begin{equation}\label{eqWU1} a_{(1)}\otimes b_{(1)}\otimes b_{(2)}S(a_{(2)})a_{(4)}\otimes a_{(3)} = a_{(1)}\otimes b_{(1)}\otimes S(a_{(2)})a_{(4)} b_{(2)}\otimes a_{(3)}\end{equation}
Applying $\varepsilon\otimes \varepsilon$ to the first two  tensorands of    both sides of \eqref{eqWU1}  we get 
\[(b\otimes\I)(S(a_{(1)})a_{(3)}\otimes a_{(2)})= (S(a_{(1)})a_{(3)}\otimes a_{(2)})(b\otimes\I)\] 
i.e.  $\cad(\sA)\subset  \sA\otimes \mathcal{Z}(\sA)$. 

The converse implication is clear, i.e. if $\cad(\sA)\subset  \sA\otimes \mathcal{Z}(\sA)$ then $\cad$ is an algebra homomorphism and the theorem is proved. 
\end{proof}

Assuming that    $\cad(\sA)\subset  \sA\otimes \mathcal{Z}(\sA)$, or equivalently that  $\cad$ is an algebra homomorphism, and  using  \Cref{easy1} we get:
\begin{corollary}Let $\sA$ be Hopf algebra and $\hz(\sA)$ its Hopf center. 
Then the   adjoint coaction  $\cad $    is an algebra homomorphism if and only if  $\cad(\sA)\subset \sA\otimes \hz(\sA)$. In particular $\cad$ may be viewed as a coaction of  $\hz(\sA)$ on $\sA$.  
\end{corollary}

We find the following observation interesting in its own.
 
\begin{lemma}\label{inval1} 
Let $\sA$ be a Hopf algebra and $x\in \sA$.   Then $\cad(x) = x\otimes\I$ if and only if 
\[\Delta(x)  = (\id\otimes S^2)(\Delta^{\textrm{op}}(x))\] 
\end{lemma}
\begin{proof}
Assuming that 
\[x\otimes \I = x_{(2)}\otimes S(x_{(1)})x_{(3)}\] we get 
\[x_{(1)}\otimes x_{(2)} \otimes\I = x_{(2)}\otimes x_{(3)}\otimes S(x_{(1)})x_{(4)}\] Thus 
\[\begin{split}x_{(2)}\otimes S^2(x_{(1)}) &= x_{(3)}\otimes S^2(x_{(2)})S(x_{(1)})x_{(4)}\\& = x_{(3)}\otimes S(x_{(1)}S(x_{(2)}))x_{(4)}\\& = x_{(1)}\otimes x_{(2)}\end{split}\] 

Conversely, if 
\[x_{(1)}\otimes x_{(2)} = x_{(2)}\otimes S^2(x_{(1)})   \] then 
\[x_{(1)}\otimes x_{(2)} \otimes x_{(3)} =x_{(2)}\otimes x_{(3)}\otimes S^2(x_{(1)})\] and  we get 
\[\begin{split}x_{(2)}\otimes S(x_{(1)})x_{(3)} &= x_{(3)}\otimes S(x_{(2)})S^2(x_{(1)})\\& = x_{(3)}\otimes S(S (x_{(1)})x_{(2)} ) \\& =x_{(2)}\otimes \I\varepsilon(x_{(1)}) = x\otimes \I\end{split}\]
\end{proof}

Let $q:\sA\rightarrow \sB $ be a Hopf morphism. Then, us proved in paragraph following \cite[Definition 2.3.8]{AD1}  $q$ is a cocentral morphism  if and only if
\[(q\otimes\id)(\cad(x)) =q(x)\otimes \I\] In what follows we provide an alternative characterization of cocentral Hopf morphisms. 
\begin{lemma}\label{cocentral}
Let $q:\sA\rightarrow \sB $ be a Hopf morphism. Then  $q$ is a cocentral  if and only if
\[(\id\otimes q)(\cad(x)) = x\otimes \I\]
\end{lemma}
\begin{proof}
Let us consider a pair of morphisms $\pi_1,\pi_2:\sA\rightarrow \sA\otimes \sB $, $\pi_1(x) = x\otimes \I$, $\pi_2(x) = \I\otimes q(x)$.
It is easy to see that $q$ is cocentral if and only if $\pi_1*\pi_2 = \pi_2*\pi_1$. Defining  a linear map $\pi_3(x) = \I\otimes S^\sB(q(x)) = \I\otimes q(S^\sA(x))$ we get 
\[\begin{split}\pi_2*\pi_3(x) &= \pi_2(x_{(1)})\pi_3(x_{(2)})\\& = \I\otimes q(x_{(1)})S^\sB(q(x_2))\\& = \I\otimes q(x_{(1)}S^\sA(x_{(2)}))\\& =\varepsilon(x) (\I\otimes\I)\end{split}\] 
Similarly 
\[\pi_3*\pi_2(x) =\varepsilon(x) (\I\otimes\I)\] Thus $\pi_2$ is convolutively invertible. In particular $q$ is cocentral if and only if 
   \[\pi_2^{-1}*\pi_1*\pi_2 = \pi_1\] i.e. if and only if 
   \[ x_{(2)}\otimes q(S^\sA(x_{(1)})x_{(3)}) = x\otimes\I
  \]
and we are done.
\end{proof}
\begin{lemma}\label{cocentral1}
 Let us define 
 \[ \sC  = \textrm{linspan}\{(\omega\otimes\id)(\cad(x)):\omega\in \sA^*,x\in \sA\}\] 
 Then \begin{itemize}
       \item $ \sC $ is a subcoalgebra of $\sA$ and $\cad(\sA)\subset \sA\otimes  \sC $;
       \item if $q:\sA\rightarrow \sB $ is cocentral then for all $x\in  \sC $ we have
       \[(q\otimes\id)(\Delta(x)) = \I\otimes x\]
      \end{itemize}

\end{lemma}
\begin{proof}
Directly from the definition of $ \sC $ we get $\cad(\sA)\subset \sA\otimes  \sC $. 
Using $\cad(x)\in \sA\otimes  \sC $ we get 
\[
\begin{split}
\Delta((\omega\otimes\id)(\cad(x))) & = (\omega\otimes\id\otimes\id)((\id\otimes\Delta)(\cad(x)))\\& = (\omega\otimes\id\otimes\id)((\cad\otimes\id)(\cad(x)))\\&  \in(\omega\otimes\id)(\cad(A))\otimes \sC \subset  \sC \otimes \sC 
\end{split}
\] which shows that $ \sC $ is a coalgebra. 

For  $\omega\in \sA^*$ and $x\in \sA$ we get 
\[
\begin{split}
(q\otimes\id)(\Delta((\omega\otimes\id)(\cad(x)))) & = (\omega\otimes\id\otimes\id)((\id\otimes q\otimes\id)((\cad\otimes\id)(\cad(x))))\\&=(\omega\otimes\id\otimes\id)(\cad(x)_{13}) =\I\otimes(\omega\otimes\id)(\cad(x))
\end{split}
\] where in the second equality we used \Cref{cocentral}. 
\end{proof}
Let $\sD$ be the algebra  generated by $ \sC $. Since $ \sC $ is a coalgebra, $\sD$ is  a bialgebra: $\Delta( \sD)\subset \sD\otimes  \sD$.  Let $\pi:\sA\rightarrow \hc(\sA)$ be the canonical quotient. Using \Cref{cocentral1} we see that 
 \[\sD\subset\{x\in \sA: (\pi\otimes\id)(\Delta(x)) = \I\otimes x\}\]
 We shall prove that $\sD$ is preserved by a certain version of the  adjoint action. 
 
\begin{lemma}\label{normalityD} Let $\sA$ be a Hopf algebra with the bijective  antipode   $S$. 
For all $x\in \sA$ and $y\in  \sD$ we have 
\[S(x_{(1)})yx_{(2)}\in \sD\]
\end{lemma}
\begin{proof}
Let $\sH$ denote the $\sA$-comodule  whose underlying vector space is $\sA$  and the comodule map  is $\cad$. Let us view   $\sA$ as    the $\sA$-comodule with the comodule map $\Delta$. Let us consider the comodules  tensor products $\sH\otimes \sA $ and  $\sA\otimes \sH$. Using \cite[Proposition 2.4]{chi_coc}
we see that $T: \sH\otimes \sA \rightarrow \sA\otimes \sH$  
\[T(h\otimes x) = x_{(1)}\otimes h x_{(2)}\] yields the comodules identification where 
\begin{equation}\label{tinv}T^{-1}(x\otimes h)  = hS^{-1}(x_{(1)})\otimes x_{(2)}\end{equation} 
In particular, denoting the corresponding comodules maps   by $\rho^{\sH\otimes \sA}$ and $\rho^{\sA\otimes \sH}$ we have 
\[\rho^{\sH\otimes \sA} = (T^{-1}\otimes\id) \circ \rho^{\sA\otimes \sH}\circ T\] In what follows  $m:\sA\otimes \sA \to \sA$ denotes the multiplication map and $\omega\in \sA^*$ a generic linear  functional. Let us consider the element $\omega(y_{(2)}) S(y_{(1)})y_{(3)}\in\sC$ and $x\in \sA$. 
We compute 
\[\begin{split}\omega(y_{(2)})&S(x_{(1)})S(y_{(1)})y_{(3)}x_{(2)} = m\left((\id\otimes\omega\otimes\varepsilon\otimes\id)(S(x_{(1)})\otimes y_{(2)}\otimes x_{(2)}\otimes S(y_{(1)})y_{(3)}x_{(3)})\right)\\&=m\left((\id\otimes\omega\otimes\varepsilon\otimes\id)(S(x_{(1)})\otimes \rho^{\sH\otimes \sA}(y\otimes x_{(2)}))\right)\\&
 = m\left((\id\otimes((\omega\otimes\varepsilon)\circ T^{-1})\otimes\id)(S(x_{(1)})\otimes \rho^{\sA\otimes \sH}(T(y\otimes x_{(2)})))\right)\\&=
m\left((\id\otimes((\omega\otimes\varepsilon)\circ T^{-1})\otimes\id)(S(x_{(1)})\otimes \rho^{A\otimes H}(  x_{(2)}\otimes y x_{(3)}))\right)
\\&=
m\left((\id\otimes((\omega\otimes\varepsilon)\circ T^{-1})\otimes\id)(S(x_{(1)})\otimes x_{(2)}\otimes((\I\otimes x_{(3)})\cad(yx_{(4)}))\right)
\end{split}\]
Denoting $\Omega = (\omega\otimes\varepsilon)\circ T^{-1}\in (\sA\otimes \sA)^*$ and using \Cref{tinv} we get 
\[\Omega(a\otimes b) =  (\omega\otimes\varepsilon)(aS^{-1}(b_{(1)})\otimes b_{(2)}) = \omega(aS^{-1}(b))\] and 
\[\begin{split}\omega(y_{(2)})S(x_{(1)})S(y_{(1)})y_{(3)}x_{(2)}&= m\left((\id\otimes\Omega\otimes\id)(S(x_{(1)})\otimes x_{(2)}\otimes((\I\otimes x_{(3)})\cad(yx_{(4)}))\right)\\&= m\left((\id\otimes\Omega\otimes\id)(S(x_{(1)})x_{(3)}\otimes x_{(2)}\otimes\cad(yx_{(4)}))\right) \end{split}\] where in the last equality,  moving $x_{(3)}$ towards $S(x_{(1)})$, we used the fact that, at the end the multiplication map $m$ is applied.  The last expression is of the form (here we use the leg numbering notation) 
\[m\left((\id\otimes\Omega\otimes\id)(\cad(x_{(1)})_{21} \cad(yx_{(2)})_{34})\right)\] Since 
$\cad(\sA)\subset \sA\otimes \sC $  we get 
\[\cad(x_{(1)})_{21} \cad(yx_{(2)})_{34}\in \sC \otimes \sA\otimes \sA\otimes \sC \] thus 
\[\omega(y_{(2)})S(x_{(1)})S(y_{(1)})y_{(3)}x_{(2)}\in \sD\] 
 \end{proof}

\begin{remark}\label{rem.free}
  Following \cite[$\S$3]{schau}, consider the Hopf algebra $\sA$ with
  bijective antipode freely generated by the $n\times n$ matrix
  coalgebra $\sC$ for some $n\ge 2$ (the notation in loc. cit. is
  $\widehat{H}(\sC)$). In this case, it can be shown that the subalgebra
  $ \sD $ defined above is not invariant under the antipode. 

Hence, it
  is possible for $\sD$ to be an ad-invariant sub-bialgebra of $\sA$
  but not a Hopf subalgebra.   
\end{remark}

\section{Cocentral exact sequence}

The following result will be crucial for the proof  that a Hopf algebra and its cocenter fit into an exact sequence in the sense of \Cref{esdef}. Before we state it, we need the following notion.

\begin{definition}\label{def.saturated}
  Let $f:\sA\to \sB$ be a Hopf algebra morphism. We say that $f$ \define{trivializes} a subset $\cS\subset \sA$ (or that $\cS$ is \define{trivialized} by $f$) if
  \begin{equation*}
    ((\id\otimes f)\circ\Delta)(s)=s\otimes 1,\ \forall s\in \cS. 
  \end{equation*}
  
  Under the above setup, the \define{Hopf reflection} of $\cS\subset \sA$ is a Hopf algebra morphism $\sA\to \sB$ that trivializes $\cS$ and is universal with this property, in the following sense:

any Hopf algebra map $\sA\to \sB'$ that trivializes $\cS$ factors uniquely as
\begin{equation*}
  \begin{tikzpicture}[auto,baseline=(current  bounding  box.center)]
    \path[anchor=base] (0,0) node (1) {$ \sA $} +(2,.5) node (2) {$\sB$} +(4,0) node (3) {$\sB'$};
         \draw[->] (1) to[bend left=10] node[pos=.5] {} (2);     
         \draw[->] (2) to[bend left=10] (3);     
         \draw[->] (1) to[bend right=10] (3);          
  \end{tikzpicture}
\end{equation*}
A Hopf algebra map $\sA\to \sB$ is \define{saturated} if it is the Hopf reflection of some subset $S\subseteq \sA$. 
\end{definition}
\begin{remark}
  Note that the Hopf reflection of a subset $\cS\subseteq \sA$ may or may not exist, but if it does it is a uniquely determined Hopf quotient of $\sA$. 
\end{remark}

With this in place, we can state the result alluded to above.

\begin{theorem}\label{saturated}
  Every saturated Hopf algebra map $f:\sA\to \sB$ invariant under either the left or right adjoint coaction of $\sA$ on itself fits into an exact sequence 
  \begin{equation*}
    \sk\to \sC\to \sA\to \sB\to \sk
  \end{equation*}
of Hopf algebras in the sense of \Cref{esdef}. 
\end{theorem}
\begin{proof}
    Let $\cS\subseteq A$ be a subset whose Hopf reflection $f$ is. Define $\sC\subset \sA$ by
  \begin{equation*}
    \sC=\{a\in \sA\ |\ ((\id\otimes f)\circ\Delta)(a)=a\otimes 1\in \sA\otimes \sB\}. 
  \end{equation*}
The adjoint coaction invariance assumption on $f$ ensures that $\sC$ is a Hopf algebra (by the dual version of \cite[Lemma 1.1.11]{AD}). Moreover, we then know from \cite[Lemma 1.1.12 (i)]{AD} that $\sA\to \sA/\sA\sC^+$ is a Hopf algebra quotient. According to \Cref{esdef} we will be done with the proof of the theorem once we argue that $\sA\sC^+=\ker(f)$ (i.e. $\sA\to \sA/\sA\sC^+$ can be identified with $f$). 

Since $\sA\to \sA/\sA\sC^+$ is the Hopf cokernel of $\sC\to \sA$ which in turn is the Hopf kernel of $f$, we already know that the latter map factors through $\sA/\sA\sC^+$, and hence $\sA\sC^+\subseteq\ker(f)$.

On the other hand $\sC$ contains $\cS$, so the universality property implicit in the saturation hypothesis shows that we have a factorization
\begin{equation*}
  \begin{tikzpicture}[auto,baseline=(current  bounding  box.center)]
    \path[anchor=base] (0,0) node (1) {$ \sA $} +(2,.5) node (2) {$\sB$} +(4,0) node (3) {$\sA/\sA\sC^+$};
         \draw[->] (1) to[bend left=10] node[pos=.5] {$f$} (2);     
         \draw[->] (2) to[bend left=10] (3);     
         \draw[->] (1) to[bend right=10] (3);          
  \end{tikzpicture}
\end{equation*}
thus getting the reverse inclusion $\ker(f)\subseteq \sA\sC^+$. 
\end{proof}

As a consequence, we get

\begin{corollary}\label{cor.gen_seq} 
    Let $\sA$ be an arbitrary Hopf algebra. Then, the cocenter $\hc=\hc(\sA)$ fits into an exact sequence
  \begin{equation}\label{eq.coc_seq}
    \sk\to \sC\to \sA\to \hc\to \sk
  \end{equation}
of Hopf algebras.
\end{corollary}
\begin{proof}
  $\sA\to \hc$ is invariant under both adjoint coactions, and it is saturated by \Cref{cocentral,cocentral1} with $\cS$ being the set of right hand tensorands of
  \begin{equation*}
     x_{(2)}\otimes S( x_{(1)})x_{(3)}\in \sA\otimes \sA
  \end{equation*}
for $x\in A$. Hence \Cref{saturated} applies to this situation and we are done. 
\end{proof}

Exact sequences are particularly pleasant when we know in addition that the Hopf algebra in the middle is faithfully (co)flat over the other terms. 
We have the following analogue of \Cref{hgc_sequence} for the cocenter (rather than the group cocenter) in case $\sA$ is pointed.

\begin{corollary}\label{cor.pointed_seq}
  Let $\sA$ be a pointed Hopf algebra. Then, the cocenter $\hc=\hc(\sA)$ fits into an exact sequence
  \begin{equation*}
    \sk\to \sC\to \sA\to \hc\to \sk
  \end{equation*}
of Hopf algebras, with $\sA$ faithfully flat over $\sC$ and faithfully coflat over $\hc$. 
\end{corollary}
\begin{proof}
  The claim on the existence of the exact sequence is \Cref{cor.gen_seq}, so only the (co)flatness claims need proof. 

According to \cite[Theorem 9.3.1]{rad_book} (or \cite{rad_free}), pointed Hopf algebras are faithfully flat (free, in fact) over their Hopf subalgebras. Hence we know that $\sA$ is faithfully flat over $\sC$, so that according to \cite[Theorem 1]{tak} in the exact sequence 
\begin{equation*}
  \sk\to \sC\to \sA\to \sA/\sA\sC^+\to \sk
\end{equation*}
$\sA$ faithfully coflat over $\sA/\sA\sC^+$ and $\cH \cong \sA\sC^+$.
\end{proof}

In fact, with a little more work we can extend this result to the fully general setup of this paper.

\begin{theorem}\label{coflat}
  Every Hopf algebra $\sA$ is faithfully flat over its cocenter $\hc$, and faithfully flat over the Hopf kernel $\sC$ of the cocenter. 
\end{theorem}

Before going into the proof, we need some preparation. The context for the next result is as follows. 
\begin{itemize}
  \item $\sA$ is an arbitrary Hopf algebra; 
  \item $\sC'\subseteq \sA$ is the Hopf kernel of the group cocenter $\sA\to \hgc$ (i.e. the left hand term of the exact sequence in \Cref{hgc_sequence});
  \item $\sC\subseteq \sC'$ is the Hopf kernel of the cocenter $\sA\to \hc$.
\end{itemize}

\begin{lemma}
  With the above notation the Hopf algebras $\sC\subseteq \sC'$ share the same coradical.
\end{lemma}
\begin{proof}
  Let $\sV$ be a simple $\sA$-comodule. When regarded as an $\hgc$-comodule, $\sV$ breaks up as a direct sum of copies of the same one-dimensional comodule corresponding to some grouplike in the universal grading group $\G(\sA)=\sG(\hgc)$ (see \Cref{hgc_sequence} and surrounding discussion).

On the other hand, since $\sV\in \cM^\sA$ is simple and the dual $\hc^*$ acts on $\sA$-comodules via $\sA$-comodule morphisms, the same reasoning as in the proof of \Cref{th.an=ch} shows that $\sV$ is also a sum of copies of the same one-dimensional $\hc$-comodule. Moreover, we know from \Cref{pr.cocentral_gplikes} that $\hc\to \hgc$ induces an isomorphism at the level of coradicals, so $\hc$ coacts trivially on $\sV$ if and only if $\hgc$ does. 

We can now finish the proof: the coradical of $\sC'$ ($\sC$) consists of those simple subcoalgebras of $\sA$ that are coefficient subcoalgebras for the simple $\sA$-comodules coacted upon trivially by $\hgc$ (respectively $\hc$).  
\end{proof}

\begin{proof}[Proof of \Cref{coflat}]
  Since we already have an exact sequence \Cref{eq.coc_seq}, the flatness and coflatness assertions are equivalent to one another by \cite[Theorems 1 and 2]{tak}. Hence, it suffices to show that $\sA$ is flat over the Hopf kernel $\sC$ of the cocenter. 

We know from \Cref{hgc_sequence} that $\sA$ is faithfully flat over the Hopf kernel $\sC'$ of the group cocenter $\sA\to \hgc$. 

On the other hand, $\sC\subseteq \sC'$ is an inclusion of Hopf algebras with a common coradical. According to \cite[Corollary 1]{rad_free}, $\sC'$ is faithfully flat (and in fact free) over $\sC$. The desired conclusion follows from the transitivity of faithful flatness. 
\end{proof}

We now turn to the issue of identifying the Hopf kernel $\sC$ of $\sA\to \hc$ more explicitly. The following general result in the context of saturated Hopf algebra morphisms will be of use in that respect.

\begin{lemma}\label{hker_aux}
  Let $f:\sA\to \sB$ be a Hopf algebra map satisfying the hypotheses of \Cref{saturated} which is the Hopf reflection of some subset $\cS\subseteq \sA$.

Let $\sD\subseteq \sA$ be a Hopf subalgebra such that 
  \begin{enumerate}
    \item[(a)] $\sD$ contains $\cS$,
    \item[(b)] $\sD$ is invariant under either the left or the right adjoint action of $\sA$ on itself and
    \item[(c)] $\sA$ is left faithfully flat over $\sD$. 
  \end{enumerate}
Then, $\sD$ contains $\sC$. Moreover, if $\sD$ is generated by $\cS$ as a Hopf algebra, then $\sD=\sC$. 
\end{lemma}
\begin{proof}
  Conditions (a) and (b) ensure that $\sA\to \sA/\sA\sD^+$ is a quotient Hopf algebra that trivializes $\cS$. By the universality of $\sA\to \sB$ implicit in being the reflection of $\cS$ (see \Cref{def.saturated}), we have a factorization 
\begin{equation}\label{eq.fact}
  \begin{tikzpicture}[auto,baseline=(current  bounding  box.center)]
    \path[anchor=base] (0,0) node (1) {$ \sA $} +(2,.5) node (2) {$\sB$} +(4,0) node (3) {$\sA/\sA\sD^+$};
         \draw[->] (1) to[bend left=10] node[pos=.5] {$f$} (2);     
         \draw[->] (2) to[bend left=10] (3);     
         \draw[->] (1) to[bend right=10] (3);          
  \end{tikzpicture}
\end{equation}
The faithful flatness hypothesis implies via \cite[Theorem 1]{tak} that $\sD$ is precisely the Hopf kernel of $\sA\to \sA/\sA\sD^+$. Since $\sC$ is the Hopf kernel of $\sA\to \sB$, the conclusion $\sC\subseteq \sD$ follows from commutative diagram \Cref{eq.fact}.  

As for the last statement, if $\sD$ is generated by $\cS$ and $\sC\supseteq \cS$ is another Hopf algebra containing $\cS$, the opposite inclusion $\sC\supseteq \sD$ also holds. 
\end{proof}

We will now seek to apply \Cref{hker_aux} to the sub-bialgebra $\sD\subseteq \sA$ defined in the discussion following \Cref{cocentral1} (generated by the right hand tensorands of the right adjoint coaction \Cref{aldef}). Taking $\cS\subseteq \sA$ to be the subspace of right hand tensorands of
\begin{equation*}
  x_{(2)}\otimes  S(x_{(1)})x_{(3)}\quad x\in \sA,
\end{equation*}
we get

\begin{corollary}\label{Dseq}
Let $\sA$ be a Hopf algebra, and suppose $\sD\subseteq \sA$ defined as above is a Hopf algebra and $\sA$ is left faithfully flat over $\sD$. Then, $\sD$ is the Hopf kernel of the cocenter of $\sA$.  
\end{corollary}
\begin{proof}
This is an immediate consequence of \Cref{hker_aux}.
\end{proof}

\begin{remark} 
  If $\sD$ is not a Hopf algebra then let $\sD'$ be the smallest Hopf subalgebra of $\sA$ containing $\sD$ and being preserved by the adjoint action. Then assuming that $\sA$ is faithfully flat over $\sD'$ and reasoning as in the proof of \Cref{Dseq} we get the exact sequence 
\begin{equation*}
 \sk\to \sD'\to \sA\to \hc \to \sk
\end{equation*} 
\end{remark}

\begin{example}
Let $\sA$ be a cosemisimple Hopf algebra and $[u^\beta_{ij}] = u^\beta$ a finite dimensional irreducible corepresentation of $\sA$. Then 
\[\cad(u^\beta_{ij}) = \sum_{k,l} u^\beta_{kl}\otimes S(u^\beta_{ik})u^\beta_{lj}\] 
Thus $\sD$ is generated by the elements $S(u^\beta_{ik})u^\beta_{lj}$ where $\beta$ runs over the equivalence classes of irreducible  corepresentations of $\sA$ and $i,j,k,l\in\{1,\ldots,\dim u^\beta\}$. Noting that 
\[S(S(u^\beta_{ik})u^\beta_{lj}) = S(u^\beta_{lj})S^2(u^\beta_{ik})\] and using the fact that $[S^2(u^\beta_{ik})]$ is  in the discussed  cosemisimple case equivalent with $u^\beta$ we see that the set generating $\sD$ is preserved by $S$. Thus  $ \sD$ is a Hopf subalgebra of $\sA$ which is preserved by the adjoint action (see \Cref{normalityD}) and since $\sA$ is faithfully flat over $\sD$ (see \cite{chi_cos}) then by \Cref{Dseq} we get  the exact sequence 
\begin{equation*}
  \sk\to \sD\to \sA\to \hc \to \sk
\end{equation*} Summarizing we get an alternative proof of \cite[Proposition 2.13]{chi_cos}. 
\end{example}

We will see below that the faithful flatness assumption in \Cref{Dseq} can essentially be dropped. In order to do this, we need some preliminaries. We will be referring to \cite{rad_book} for background on comodule theory. Specifically, we recollect some notions from \cite[Sections 3.7, 4.8]{rad_book}.

\begin{definition}\label{def.link}
  Two simple subcoalgebras $\sC,\sD$ of a coalgebra $\sA$ are \define{directly linked} if the space $\sC\wedge \sD+\sD\wedge \sC\subseteq \sA$ is strictly larger than $\sC+\sD$. 

$\sC$ and $\sD$ are \define{linked} if there is a finite sequence of simple subcoalgebras starting with $\sC$ and ending with $\sD$, so that every two consecutive ones are directly linked. 

A coalgebra is \define{link indecomposable} if all of its simple subcoalgebras are linked. 

A \define{link indecomposable component} (or just `component' for short) of a coalgebra is a maximal link indecomposable subcoalgebra. 
\end{definition}
\begin{remark}\label{re.comp=block}
If $\sV$ and $\sW$ are the simple comodules corresponding to the simple subcoalgebras $\sC$ and $\sD$ respectively, then $\sC$ and $\sD$ are directly linked if and only if there is an indecomposable $\sA$-comodule admitting a composition series whose components are $\sV$ and $\sW$. In other words, there is a non-split exact sequence
\begin{equation}\label{eq.VW}
  0\to \sV\to \bullet\to \sW\to 0
\end{equation}
or 
\begin{equation}\label{eq.WV}
  0\to \sW\to \bullet\to \sV\to 0
\end{equation}
in $\cM^\sA$. More precisely, \Cref{eq.VW} corresponds to $\sC\oplus \sD\subsetneq \sC\wedge \sD$  while the existence of a non-split sequence \Cref{eq.WV} is equivalent to $\sC\oplus \sD\subsetneq \sD\wedge \sC$. 

Every coalgebra breaks up as the direct sum of its components. In other words, the components of a coalgebra $\sA$ correspond to the blocks of the category $\cM^\sA$. 
\end{remark}

We now introduce the following notion.

\begin{definition}\label{def.replete}
  Let $\sA$ be a coalgebra. A subcoalgebra $\sC\subseteq \sA$ is \define{coradically replete} if for any two linked simple subcoalgebras $\sD,\sE\in \sA$ the inclusion $\sD\subseteq \sC$ implies $\sE\subseteq \sC$.

$\sC$ is \define{replete} if it is a direct summand of $\sA$ as a coalgebra (i.e. $\sC$ is a direct sum of some of the components of $\sA$).  
\end{definition}

We are now ready to state a number of auxiliary results.

\begin{lemma}\label{le.split}
  Let $\sA$ be a Hopf algebra and $\sC\subseteq\sA$ a replete Hopf subalgebra. Then, the inclusion $\sC\to \sA$ splits as a map of $\sC$-bimodules and $\sA$-bicomodules. 
\end{lemma}
\begin{proof}
  By hypothesis, we have a decomposition $\sA=\sC\oplus \sD$ as coalgebras. We will be done if we argue that $\sD$ is a $\sC$-bimodule, since in that case the required splitting will simply be the projection $\sA=\sC\oplus \sD\to \sC$. 

Our task amounts to showing that if $\sV$ and $\sW$ are comodules over
$\sC$ and $\sD$ respectively, then their tensor product in $\cM^\sA$
does not contain any $\sC$-subcomodules. Since comodules are unions of
their finite-dimensional subcomodules, we will assume that $\sV$ and $\sW$
are finite-dimensional.  

Suppose we have some non-zero map $\sU\to \sV\otimes \sW$ for $\sU\in \cM^\sC$. This means we have a non-zero map $\sV^*\otimes \sU\to \sW$,  contradicting the fact that $\sV^*\otimes \sU$ is a $\sC$-comodule (because the latter is a Hopf subalgebra of $\sA$).  

On the other hand, a non-zero map $\sU\to \sW\otimes \sV$ would provide a non-zero morphism $\sW^*\to \sV\otimes \sU^*$, and we can repeat the argument. 
\end{proof}

As a consequence, we get

\begin{proposition}\label{cor.replete_ff}
  A Hopf algebra is left and right faithfully flat over any replete Hopf subalgebra. 
\end{proposition}
\begin{proof}
  The splitting in \Cref{le.split} implies faithful flatness as in the proof of the main result in \cite{chi_cos}.
\end{proof}

We can strengthen this as follows.

\begin{theorem}\label{th.corad_replete_ff}
  A Hopf algebra is left and right faithfully flat over any coradically replete Hopf subalgebra. 
\end{theorem}
\begin{proof}
Let $\sA$ be a Hopf algebra and $\sC\subseteq \sA$ a coradically replete Hopf subalgebra. 

Consider the subcoalgebra $\sD\subseteq \sA$ defined as the sum of all components of $\sA$ that intersect $\sC$ non-trivially. $\sD$ is a Hopf subalgebra, and $\sC\subseteq \sD\subseteq \sA$.
As in the proof of \Cref{coflat}, the left hand inclusion is faithfully flat by \cite[Corollary 1]{rad_free} because the two Hopf algebras share the same coradical. On the other hand, $\sD\subseteq \sA$ is faithfully flat by \Cref{cor.replete_ff}.
\end{proof}

\Cref{th.corad_replete_ff} is relevant to us because of the following

\begin{lemma}\label{le.D_corad_replete}
  Let $\sA$ be a Hopf algebra with bijective antipode. If the sub-bialgebra $\sD\subseteq \sA$ generated by the right hand tensorands of \Cref{aldef} is a Hopf subalgebra with bijective antipode, then it is coradically replete.   
\end{lemma}
\begin{proof}
We need to show that if $\sC\subseteq \sD$ and $\sE\subseteq \sA$ are directly linked simple subcoalgebras, then $\sE$ too is contained in $\sD$. 

According to the definition of being directly linked (see \Cref{def.link}), there are two possibilities we have to treat. 

{\bf Case 1: $\sC\oplus \sE\subsetneq \sC\wedge \sE$.} Let $\sV$ and $\sW$ be the simple right comodules over $\sC$ and $\sE$ respectively   entering \Cref{eq.VW}. Our hypotheses then ensure via \Cref{le.coend} below that $\sV^*\otimes \sW$ is a $\sD$-comodule. Since $\sV$ is also a $\sD$-comodule by assumption and $\sW$ is a quotient of $\sV\otimes \sV^*\otimes \sW$, we have $\sW\in \cM^\sD$ and hence $\sE\subseteq \sD$.

{\bf Case 2: $\sC\oplus \sE\subsetneq \sE\wedge \sC$.} Keeping the above notation for $\sV$ and $\sW$, this time around we know that $\sW^*\otimes \sV$ is a $\sD$-comodule, along with $\sV$. This means first that $\sW^*$, being a quotient of $\sW^*\otimes \sV\otimes \sV^*$, must be a $\sD$-comodule. In turn, this implies $\sW\in \cM^\sD$ because $\sD$ is preserved by the inverse of the antipode (this is where that hypothesis is necessary). 
\end{proof}

\begin{lemma}\label{le.coend}
  Let $\sC\subseteq \sA$ be the subcoalgebra spanned by the right hand tensorands of the adjoint coaction  \Cref{aldef}. Suppose $\sV$ and $\sW$ are simple $\sA$-comodules for which there is a non-split exact sequence 
  \begin{equation}\label{eq.nonsplit}
     0\to \sV\to \sU\to \sW\to 0 
  \end{equation}
in $\cM^\sA$. Then, $\sV^*\otimes \sW$ is a $\sC$-comodule.  
\end{lemma}
\begin{proof}
Consider the canonical coalgebra map $\phi:\sU^*\otimes \sU\to \sA$ implementing the $\sA$-coaction on $\sU$. The fact that \Cref{eq.nonsplit} does not split translates to the fact that $\phi$ factors through an embedding into $\sA$ of the quotient coalgebra 
\begin{equation}\label{eq.matr}
\sE= \begin{pmatrix}
    \sV^*\otimes \sV & \sV^*\otimes \sW \\
    0            & \sW^*\otimes \sW
  \end{pmatrix}
\end{equation}
of $\sU^*\otimes \sU$ (in other words, its quotient by $\sW^*\otimes \sV$). 

If we think of $\sU^*\otimes \sU$ as an object of $\cM^\sA$ with the tensor product comodule structure, then it is a coalgebra in the monoidal category $\cM^\sA$, and its coaction on $\sU$ is a morphism of coalgebras. Additionally, the map $\phi$ is one of $\sA$-comodules if we give $\sA$ its right adjoint coaction, and so $\sE$ is an $\sA$-comodule coalgebra coacting on $\sU$ so that the comodule structure map $\sU\to \sU\otimes \sE$ is one of $\sA$-comodules. As $x$ ranges over the image $\sE$ of $\phi$, the right hand tensorands of the adjoint coaction \Cref{aldef} span the coefficient coalgebra of $\sU^*\otimes \sU\in \cM^\sA$. In conclusion, $\sE\in \cM^\sA$ is actually a $\sC$-comodule. 

As the subcategory $\cM^\sC\subseteq \cM^\sA$ is closed under taking subobjects (it is closed in the sense of \Cref{def.closed}), the $\sA$-comodule coalgebras $\sV^*\otimes \sV\subseteq \sE$ and $\sW^*\otimes \sW\subseteq \sE$ are both $\sC$-comodules. In conclusion, so is the quotient $\sV^*\otimes \sW$ of $\sE$ by their sum. 
\end{proof}

\begin{remark}
  Note that when $\sA$ has bijective antipode, the antipode bijectivity of $\sD$ is necessary for faithful flatness (e.g. by \cite[Theorem 3.2 (b)]{chi_epi}); \Cref{le.D_corad_replete} reverses this implication.
\end{remark}

Finally from \Cref{le.D_corad_replete,th.corad_replete_ff}, we get

\begin{corollary}
Under the assumptions of \Cref{le.D_corad_replete} $\sA$ is left and right faithfully flat over $\sD$.  
 \qedhere 
\end{corollary}




\section{Examples}

\subsection{The center and cocenter of $q$-deformations  of  $\bG$}
Let $\bG$ be a compact semisimple simply connected  Lie group,  $\GG_q$ the Drinfeld-Jimbo quantization of $\bG$ for some transcendental $q\in \bC^\times$, and $\sA = \textrm{Pol}(\GG_q)$. Using  \cite[Theorem 9.3.20]{Joseph} we see that $\cZ(\sA) = \bC$  thus $\hz(\sA) = \bC$. 

Since the cocenter of a semisimple Hopf algebra $\textrm{Pol}(\GG_q)$ depends only on the fusion ring and the fusion ring of $\textrm{Pol}(\GG_q) $ is the same as the fusion ring of $\textrm{Pol}(\bG)$ then we see that $\hc(\textrm{Pol}(\GG_q)) = \hc(\textrm{Pol}(\bG))$. The later is the group algebra of the center of $\bG$.

\begin{remark}
The situation is more complicated when $q$ is a root of unity. In that case, the appropriately-defined $\textrm{Pol}(\bG_q)$ contains a central Hopf subalgebra isomorphic to $\textrm{Pol}(\bG)$, and the resulting exact sequences can be applied to the study and classification of finite quantum subgroups of $\bG_q$; see e.g. \cite{dCL,AG}. 
\end{remark}

\subsection{Drinfeld twist}
Let $\sA$ be a Hopf algebra and $\Psi\in \sA\otimes \sA$ an invertible 2-cocycle. Let $\Delta^\Psi:\sA\rightarrow \sA\otimes \sA$ be the twisted comultiplication \[\Delta^\Psi(x) = \Psi\Delta(x)\Psi^{-1}\] It is well known that $\sA$ equipped with $\Delta^\Psi$ is a Hopf algebra which we denote by  $\sA^\Psi$. Using \Cref{easy} we get $\hz(\sA^\Psi )= \hz(\sA )$.

\subsection{Center of $\sU_q(\fg)$}
Let us consider the quantized enveloping algebra $\sU_q(\mathfrak g)$, defined, say, as in \cite[$\S$6.1.2]{KS}. Here we are again working over the complex numbers, $q\in \bC^\times$ and $\fg$ is a simple complex Lie algebra of rank $\ell$. 

Let $\alpha_i$, $1\le i\le \ell$ be a set of simple roots for $\fg$ and $(a_{ij})$ the associated Cartan matrix. Finally, we set $d_i=(\alpha_i,\alpha_i)/2$ for an appropriate symmetric bilinear form on the space spanned by $\alpha_i$ and define $q_i=q^{d_i}$.  
$\sU_q(\fg)$ then has generators 
\begin{equation*}
  E_i,\ F_i,\ K_i,\quad 1\le i\le \ell
\end{equation*}
such that $K_i$ are invertible and commute with one another, 
\begin{equation}\label{eq.qrels}
  K_iE_jK_i^{-1}=q_i^{a_{ij}}E_j,\quad K_iF_jK_i^{-1}=q_i^{-a_{ij}}F_j,\quad [E_i,F_j] = \delta_{ij}\frac{K_i-K_i^{-1}}{q_i-q_i^{-1}}
\end{equation}
(where we are assuming $q^{2d_i}\ne 1$) and some additional relations are satisfied that are meant to mimic the Serre relations for $\fg$. We refer to loc. cit. for details. 

The algebra $\sU_q(\fg)$ has a unique Hopf algebra structure making $K_i$ grouplike and for which
\begin{equation}\label{eq.coprod}
  \Delta(E_i) = E_i\otimes K_i+1\otimes E_i,\quad \Delta(F_i) = F_i\otimes 1+K_i^{-1}\otimes F_i. 
\end{equation}

\begin{proposition}\label{pr.noroot}
If $q$ is not a root of unity then the Hopf center $\hz=\hz(\sU_q(\fg))$ is $\bC$. 
\end{proposition}
\begin{proof}
  Note that $\sU=\sU_q(\fg)$ is \define{pointed} as a coalgebra (\cite[p. 157]{Sweedler}) in the sense that its simple comodules are one-dimensional (they correspond to the grouplikes, which make up the group generated by the $K_i$). Since $\hz$ is a Hopf subalgebra of $\sU$, it too must be pointed.

{\bf Claim 1: $\hz$ is irreducible as a coalgebra, i.e. $1$ is its only grouplike.} Indeed, the set of grouplikes is the group
\begin{equation*}
  \langle K_i,\ 1\le i\le \ell\rangle \cong \bZ^\ell. 
\end{equation*}
A non-trivial element of it is of the form $g=\prod_i K_i^{n_i}$ with
at least one non-zero $n_i$. Fixing such an $i$, the conjugate
$gE_ig^{-1}$ is $q^{(\lambda,\alpha_i)}E_i$, where
$\lambda=\sum n_i\alpha_i$. Now, $(\lambda,\alpha_i)$ cannot be zero
for all $i$ unless $\lambda$ itself is zero, because of the
non-degeneracy of $(-,-)$. In addition, if $(\lambda,\alpha_i)\ne 0$,
then  
\begin{equation*}
  gE_ig^{-1}=q^{(\lambda,\alpha_i)}E_i\ne E_i
\end{equation*}
because $q$ is not a root of unity. In conclusion, the centrality of
$g$ implies $n_i=0$ for all $i$. This proves the claim.   

Now suppose $\hz$ is not trivial. Then, since it is pointed irreducible, it must contain non-zero \define{primitive} elements, i.e. $x\in \hz$ such that 
\begin{equation*}
  \Delta(x) = x\otimes 1+1\otimes x 
\end{equation*}
(e.g. \cite[Corollary 11.0.2]{Sweedler}). We will have reached the desired contradiction once we prove

{\bf Claim 2: $\sU$ has no non-zero central primitive elements.} Let $0\ne x\in \hz$ be a primitive element. Being central, $x$ acts as a scalar $c_\lambda$ on each simple $\sU$-module of type $\sV_\lambda$ corresponding to some dominant weight $\lambda$ of $\fg$. Since $\sU$ acts jointly faithfully on all $\sV_\lambda$, $x\ne 0$ implies that $c_\lambda\ne 0$ for at least some $\lambda$. 

Recall e.g. from \cite[$\S$6.2.3]{KS} that $\sU$ has a PBW basis, in the sense that the elements 
\begin{equation}\label{eq.pbw}
  F_{\beta_1}^{r_1}\cdots F_{\beta_n}^{r_n}K_1^{t_1}\cdots K_\ell^{t_\ell}E_{\beta_n}^{s_n}\cdots E_{\beta_1}^{s_1}
\end{equation}
form a vector space basis. Here, the $r_i$ and $s_i$ are non-negative integers, the $t_i$ are integers, and $\beta_1$ up to $\beta_n$ are the positive roots of $\fg$ ordered in a certain way; see loc. cit. for details. 

Now, $x$ is a linear combination of elements of the form \Cref{eq.pbw}. Under this expansion, those summands that contain at least some $E$s annihilate a highest weight vector $v_\lambda$ for every $\sV_\lambda$. Since there are at least some $\lambda$ for which $v_\lambda$ is not annihilated, we must have summands in $x$ that contain no $E$s. But these summands can then contain no $F$s either, because elements of the form \Cref{eq.pbw} with $F$s but no $E$s will strictly lower weights, whereas $x$, being central, cannot. 

In conclusion, when expanded according to the basis \Cref{eq.pbw} $x$ contains summands of the form
\begin{equation}\label{eq.pbw_k}
  K_1^{t_1}\cdots K_\ell^{t_\ell}, 
\end{equation}
while all other summands contain both $E$s and $F$s. 

Now let $x_{\cat{res}}$ be the sum of all summands \Cref{eq.pbw_k} of $x$, with the respective coefficients (so $x_{\cat{res}}$ is obtained from $x$ by simply dropping the summands containing $E$s and $F$s).  

Now, both the algebra $\sU_+$ generated by the $K$s and $E$s and the algebra $\sU_-$ generated by $K$s and $F$s are Hopf subalgebras of $\sU$. They are graded as Hopf algebras via
\begin{equation*}
  \deg(K)=0,\ \deg(E)=1\text{ for }U_+,\quad \deg(K)=0,\ \deg(F)=1\text{ for }U_-
\end{equation*}
This means that when applying the coproduct to the summands in $x-x_{\cat{res}}$ and expand with respect to the tensor product of PBW basis of $\sU\otimes\sU$, we obtain only terms that contain $E$s and $F$s. In conclusion, if $\sU_0\subset \sU$ is the algebra generated by the $K$s, then the only summands of $\Delta(x)=x\otimes 1+1\otimes x$ that belong to $\sU_0\otimes \sU_0$ come from $\Delta(x_{\cat{res}})$. 

It follows from the discussion above that $x_{\cat{res}}$ itself must be a primitive element. This is impossible, since it belongs to the group algebra $\sU_0$ which has no non-zero primitives.
\end{proof}

The next result complements \Cref{pr.noroot} for (some) roots of unity, but we need a preparation before stating it. In both the discussion preceding and the proof of \Cref{pr.root} below we will follow  \cite[Section 6.3.5]{KS}.

Let $q\in\mathbb{C}$ be a primitive $p$th root of unity for some odd integer $p>d_i$, $i=1,2,\ldots,l$. For a root $\alpha $ we consider $E_\alpha,F_\alpha$, as in the proof of \Cref{pr.noroot}. Then \cite[Proposition 47]{KS}  says that 
\begin{equation*}
  e_\alpha:=E_\alpha^p,\ f_\alpha:=F_\alpha^p \text{ and } k_i=K_i^p
\end{equation*}
are central in $\sU_q=\sU_q(\fg)$. For the simple root $\alpha_i$ we shall also write $e_i = E_{\alpha_i}^p$, $f_i = F_{\alpha_i}^p$. Then, according to \cite[Proposition 48]{KS}  the algebra $\zeta_0$ generated by $e_\alpha$, $f_\alpha$ and $k_i$ is a Hopf subalgebra of $U$. 

In particular, the Hopf center of $\sU_q$ is non-trivial and it contains $\zeta_0$. With this in place, we can now state

%
%
%

\begin{proposition}\label{pr.root}
  When $q\in\bC^\times$ is a root of unity whose order $p$ is odd and larger than all $d_i$, the Hopf center of $\sU_q$ is the algebra $\zeta_0$ defined above. 
\end{proposition}
\begin{proof}
Consider the inclusion $\zeta_0\subseteq\hz=\hz(\sU_q)$ of commutative Hopf algebras. 

We know from \cite[Theorem 4.9]{KS} that $\sU_q$ is a finitely-generated free module over $\zeta_0$, and hence $\hz$ is finitely generated as a module over $\zeta_0$ (e.g. because the latter is a finitely generated commutative ring, and hence noetherian). It follows from this that in the exact sequence 
\begin{equation*}
  \bC\to \zeta_0\to \hz\to \bullet\to \mathbb{C}
\end{equation*}
of Hopf algebras the third term $\bullet$ is finite-dimensional. Since it is also commutative and we are working in characteristic zero, it must be the function algebra $\bC(\sG)$ of some finite group $\sG$. 

It is easy to see that the central grouplikes in $\sU_q$ are precisely those in the group generated by the $k_i=K_i^p$, and hence the inclusion $\zeta_0\subseteq \hz$ is an isomorphism at the level of coradicals. But this means that the cosemisimple Hopf algebra $\bC(\sG)$ is pointed irreducible, meaning that $\sG$ is trivial. 
\end{proof}

\begin{remark}
It seems likely that the determination of the centers of two-parameter quantum groups $\sU_{r,s}(\mathfrak{sl}_n)$ carried out in \cite{BKL} can be put to similar use, though we will not attempt this here. 
\end{remark}

\subsection{Cocenter of $\sU_q(\fg)$}

In what follows we denote $ \sC $   the coalgebra  assigned to $\sU_q(\fg)$ as described in general in \Cref{cocentral1}; $ \sD $ is the bialgebra generated by $ \sC $. 
\begin{lemma}\label{cocenU}
 Let $\pi:\sU_q(\fg)\rightarrow \sB $ be a cocentral morphism of  Hopf algebras. Then $\pi(x) = \varepsilon(x)\I$. In particular $\hc(\sU_q(\fg)) = \bC$. 
\end{lemma}
\begin{proof}
The cocentrality of $\pi$ yields
\[\pi(E_i)\otimes K_i + \I\otimes E_i = \pi(K_i)\otimes E_i + \pi(E_i)\otimes \I\]
Now the linear independence of $K_i,\I,E_i$ implies that $\pi(E_i) = 0$ and $\pi(K_i) = \I$. Similarly we are reasoning for $F_i=0$ and we are done.
\end{proof}
The above result shows that the Hopf kernel assigned to the surjection  $\sU_q(\fg)\to\hc(\sU_q(\fg))$ is equal $\sU_q(\fg)$. We shall show that already  $\sD=\sU_q(\fg)$. 
Computing in what follows we shall drop the index $i$:
\[(\Delta\otimes\id)(\Delta(E)) = E\otimes K \otimes K+ \I\otimes E\otimes K + \I\otimes \I\otimes E\] Thus 
\begin{equation}\label{expr1}\cad(E) = -K\otimes E + E\otimes K+ \I\otimes E\end{equation}
The linear independence of $E,K,\I$  enables us to slice the first leg of \Cref{expr1}  by functionals which put $K$ to $1$ and $E,\I$ to zero. Thus we get  
$E \in \sD $. Similarly $K\in \sD $. 
The same applied to $F$ yields 
  $F \in  \sD $.
Summarizing we proved
\begin{lemma}
 Adopting the above notation we have $ \sD  = \sU_q(\fg)$.
\end{lemma}

\subsection{General concluding remarks}
The following simple observation can be useful in computing Hopf cocenters. Before stating it, we need the following

\begin{definition}
  The \define{cocenter} of a coalgebra $ \sC $ is the finite dual of the center of $ \sC ^*$. 

We denote the cocenter of $ \sC $ by $\cz( C )$. 
\end{definition}
\begin{remark}
  Note that we have a surjection $ \sC \to \cz( \sC )$, universal among morphisms of coalgebras defined on $ \sC $ that are cocentral in the same sense as in \Cref{def.hopf_cocentral_morphism}. 
\end{remark}

\begin{lemma}\label{le.coalg_gen}  
  \begin{enumerate}
    \item[(a)] If a Hopf algebra $\sA$ is generated by a subcoalgebra $ \sC $ as a Hopf algebra, then the kernel of 
      \begin{equation*}
        \sA\to \hc(\sA)
      \end{equation*}
      is the Hopf ideal generated by the kernel of $ \sC \to \cz( \sC )$. 
    \item[(b)] The analogous statement holds for Hopf algebras with bijective antipode. 
  \end{enumerate}
\end{lemma}
\begin{proof}
  The difference between parts (a) and (b) is that in the former case $\sA$ is generated as an algebra by the iterations $S^n( \sC )$ for $n\in \bZ_{\ge 0}$, whereas in the latter we allow $n$ to range over all integers. To fix ideas, we focus on part (a); the proof for (b) is completely analogous. 

By the universality of $ \sC \to \cz$ we have a commutative diagram
\begin{equation*}
  \begin{tikzpicture}[auto,baseline=(current  bounding  box.center)]
    \path[anchor=base] (0,0) node (1) {$ \sC $} +(2,.5) node (2) {$\sA$} +(2,-.5) node (3) {$\cz$} +(4,0) node (4) {$\hc$};
         \draw[->] (1) to[bend left=10] (2);     
         \draw[->] (2) to[bend left=10] (4);     
         \draw[->] (1) to[bend right=10] (3);     
         \draw[->] (3) to[bend right=10] (4);     
  \end{tikzpicture}
\end{equation*}
of coalgebra maps, and hence the kernel of $ \sC \to \cz$ is indeed contained in that of $\sA\to \hc$. 

On the other hand, consider the Hopf ideal $\cI$ in the statement (generated by $\ker( \sC \to \cz)$). The quotient $\sA\to \sA/\cI$ is cocentral in the sense of \Cref{def.hopf_cocentral_morphism} because $ \sC \to \sA/\cI$ is, the iterations $S^n( C )$ generate $\sA$ as an algebra, and the condition from \Cref{def.hopf_cocentral_morphism} is preserved under taking products and antipodes (i.e. if it holds for $x,y\in \sA$ it also holds for $xy$, etc.).   

It follows from the previous paragraph that $\sA\to \sA/\cI$ factors through $\sA\to \hc$, and hence
\begin{equation*}
  \ker(\sA\to \hc)\subseteq \cI.
\end{equation*}
Since the preceding discussion argues that the opposite inclusion holds, we are done. 
\end{proof}

\begin{remark}
  Note that \Cref{cocenU} can be read as a particular instance of \Cref{le.coalg_gen}. 
\end{remark}

\Cref{le.coalg_gen} allows us to compute the cocenters of Hopf algebras freely generated by coalgebras in the sense of \cite{tak_free}. We omit the proof of the following result, as it is almost tautological.

\begin{corollary}
  The Hopf cocenter of the Hopf algebra $\cH( \sC )$ freely generated by a coalgebra $ \sC $ is the free Hopf algebra $\cH(\cz)$ on the cocenter $\cz=\cz( \sC )$. 
\qedhere
\end{corollary}



\begin{thebibliography}{10}

\bibitem{AD1}
N.~Andruskiewitsch.
\newblock {Notes on extensions of Hopf algebras}.
\newblock {\em Can. J. Math.}, 48(1):3--42, 1996.

\bibitem{AD}
N.~Andruskiewitsch and J.~Devoto.
\newblock {Extensions of Hopf algebras}.
\newblock {\em Algebra i Analiz}, 7(1):22--61, 1995.

\bibitem{AG}
N.~Andruskiewitsch and G.~A. Garc{\'{\i}}a.
\newblock Extensions of finite quantum groups by finite groups.
\newblock {\em Transform. Groups}, 14(1):1--27, 2009.

\bibitem{ArkGait}
Sergey Arkhipov and Dennis Gaitsgory.
\newblock {Another realization of the category of modules over the small
  quantum group.}
\newblock {\em Adv. Math.}, 173(1):114--143, 2003.

\bibitem{BL}
Hellmut Baumg{\"a}rtel and Fernando Lled{\'o}.
\newblock Duality of compact groups and {H}ilbert {$C^*$}-systems for
  {$C^*$}-algebras with a nontrivial center.
\newblock {\em Internat. J. Math.}, 15(8):759--812, 2004.

\bibitem{BKL}
Georgia Benkart, Seok-Jin Kang, and Kyu-Hwan Lee.
\newblock On the centre of two-parameter quantum groups.
\newblock {\em Proc. Roy. Soc. Edinburgh Sect. A}, 136(3):445--472, 2006.

\bibitem{BrWi}
Tomasz Brzezi\'nski and Robert Wisbauer.
\newblock {\em Corings and comodules}, volume 309 of {\em London Mathematical
  Society Lecture Note Series}.
\newblock Cambridge University Press, Cambridge, 2003.

\bibitem{BudzKasp}
M.~Budzi\'{n}ski and P.~Kasprzak.
\newblock {On the quantum families of quantum group homomorphisms}.
\newblock {\em Accepted to  Commun.~Algebra}.

\bibitem{chi_epi}
Alexandru Chirv{\u{a}}situ.
\newblock On epimorphisms and monomorphisms of {H}opf algebras.
\newblock {\em J. Algebra}, 323(5):1593--1606, 2010.

\bibitem{chi_coc}
Alexandru Chirv{\u{a}}situ.
\newblock Centers, cocenters and simple quantum groups.
\newblock {\em J. Pure Appl. Algebra}, 218(8):1418--1430, 2014.

\bibitem{chi_cos}
Alexandru Chirv{\u{a}}situ.
\newblock Cosemisimple {H}opf algebras are faithfully flat over {H}opf
  subalgebras.
\newblock {\em Algebra Number Theory}, 8(5):1179--1199, 2014.

\bibitem{ComOst11}
Jonathan Comes and Victor Ostrik.
\newblock On blocks of {D}eligne's category {$\underline{\rm Re}{\rm p}(S_t)$}.
\newblock {\em Adv. Math.}, 226(2):1331--1377, 2011.

\bibitem{dCL}
Corrado De~Concini and Volodimir Lyubashenko.
\newblock Quantum function algebra at roots of {$1$}.
\newblock {\em Adv. Math.}, 108(2):205--262, 1994.

\bibitem{Hopf}
Heinz Hopf.
\newblock {{\"U}ber die Topologie der Gruppen-Mannigfaltigkeiten und ihre
  Verallgemeinerungen.}
\newblock {\em Ann. of Math.}, 42:22--52, 1941.

\bibitem{Joseph}
A~Joseph.
\newblock {\em {Quantum groups and their primitive ideals}}.
\newblock Springer.

\bibitem{KSS}
P.~Kasprzak, A.~Skalski, and P.M. So{\l}tan.
\newblock {Short exact sequence $\{e\}\to Z(\GG)\to
  \GG\to\textrm{Inn}(\GG)\to\{e\}$ for locally compact quantum groups}.
\newblock {\em arXiv:1508.02943 preprint}.

\bibitem{KS}
Anatoli Klimyk and Konrad Schm{\"u}dgen.
\newblock {\em Quantum groups and their representations}.
\newblock Texts and Monographs in Physics. Springer-Verlag, Berlin, 1997.

\bibitem{Mu}
Michael M{\"u}ger.
\newblock On the center of a compact group.
\newblock {\em Int. Math. Res. Not.}, (51):2751--2756, 2004.

\bibitem{rad_free}
David~E. Radford.
\newblock Pointed {H}opf algebras are free over {H}opf subalgebras.
\newblock {\em J. Algebra}, 45(2):266--273, 1977.

\bibitem{rad_book}
David~E. Radford.
\newblock {\em Hopf algebras}, volume~49 of {\em Series on Knots and
  Everything}.
\newblock World Scientific Publishing Co. Pte. Ltd., Hackensack, NJ, 2012.

\bibitem{schau_tann}
Peter Schauenburg.
\newblock {\em Tannaka duality for arbitrary {H}opf algebras}, volume~66 of
  {\em Algebra Berichte [Algebra Reports]}.
\newblock Verlag Reinhard Fischer, Munich, 1992.

\bibitem{schau}
Peter Schauenburg.
\newblock Faithful flatness over {H}opf subalgebras: counterexamples.
\newblock In {\em Interactions between ring theory and representations of
  algebras ({M}urcia)}, volume 210 of {\em Lecture Notes in Pure and Appl.
  Math.}, pages 331--344. Dekker, New York, 2000.

\bibitem{Sweedler}
Moss~E. Sweedler.
\newblock {\em Hopf algebras}.
\newblock Mathematics Lecture Note Series. W. A. Benjamin, Inc., New York,
  1969.

\bibitem{tak_free}
Mitsuhiro Takeuchi.
\newblock Free {H}opf algebras generated by coalgebras.
\newblock {\em J. Math. Soc. Japan}, 23:561--582, 1971.

\bibitem{tak_form}
Mitsuhiro Takeuchi.
\newblock Formal schemes over fields.
\newblock {\em Comm. Algebra}, 5(14):1483--1528, 1977.

\bibitem{tak}
Mitsuhiro Takeuchi.
\newblock Relative {H}opf modules---equivalences and freeness criteria.
\newblock {\em J. Algebra}, 60(2):452--471, 1979.


\end{thebibliography}
\bibliographystyle{plain}

\end{document}